\documentclass{amsart}
\usepackage{amssymb}
\usepackage[v2,cmtip]{xy}

\DeclareMathAlphabet{\scr}{U}{rsfs}{m}{n}

\def\draftdate{January 1, 2011}

\newcommand{\LC}[1][n]{\oC_{#1}}
\newcommand{\EnR}[1][n]{$E_{#1}$ $R$}
\newcommand{\EnH}[1][n]{$E_{#1}$ $H$}

\chardef\bs=`\\ 
\let\overto\xrightarrow

\newcommand{\h}{\mathop{\mathrm{Ho}}\nolimits}

\newcommand{\ARH}{\aA\LC^{R/H}}
\newcommand{\AHH}{\aA\LC^{H/H}}
\newcommand{\NH}{\aN\LC^{H}}
\newcommand{\QC}[1][n]{H_{\LC[#1]}}
\newcommand{\QH}[1][n]{H^{\LC[#1]}}
\newcommand{\NC}[1][n]{\widetilde\oC_{#1}}

\newcommand{\tB}{\tilde B}

\newcommand{\lsmaR}{\sma^{\mathbf{L}}_{R}}

\newcommand{\rI}{I^{\mathbf{R}}}
\newcommand{\lQ}{Q^{\mathbf{L}}}

\newcommand{\ps}[2]{#1[#2]}

\newcommand{\del}{\partial}

\mathchardef\varDelta="7101
\newcommand{\DDelta}{{\mathbf \varDelta}}
\DeclareMathOperator{\sd}{\mathrm{sd}_{2}}
\DeclareMathOperator{\Sd}{\mathrm{Sd}_{2}}

\newcommand{\ssdot}{\bullet}
\newcommand{\subdot}{_{\ssdot}}

\newcommand{\bF}{{\mathbb{F}}}
\newcommand{\bZ}{{\mathbb{Z}}}

\let\catsymbfont\mathfrak 

\newcommand{\aA}{{\catsymbfont{A}}}
\newcommand{\aD}{{\catsymbfont{D}}}
\newcommand{\aN}{{\catsymbfont{N}}}

\let\opsymbfont\mathcal 

\newcommand{\oC}{{\opsymbfont{C}}}
\newcommand{\oQ}{{\opsymbfont{Q}}}

\newcommand{\iso}{\cong}     
\newcommand{\sma}{\wedge}    

\renewcommand{\to}{\mathchoice{\longrightarrow}{\rightarrow}{\rightarrow}{\rightarrow}}
\newcommand{\from}{\mathchoice{\longleftarrow}{\leftarrow}{\leftarrow}{\leftarrow}}

\def\quickop#1{\expandafter\DeclareMathOperator\csname #1\endcsname{#1}}
\quickop{Id}\quickop{id}\quickop{Tor}\quickop{Ext}\quickop{Hom}

\DeclareMathOperator*{\colim}{colim}

\newtheorem{thm}[equation]{Theorem}
\newtheorem{cor}[equation]{Corollary}
\newtheorem{lem}[equation]{Lemma}
\newtheorem{prop}[equation]{Proposition}

\theoremstyle{definition}
\newtheorem{defn}[equation]{Definition}

\theoremstyle{remark}
\newtheorem{rem}[equation]{Remark}

\numberwithin{equation}{section}

\newcommand{\term}[1]{\emph{#1}}

\long\def\ignore#1\endignore{\relax}

\begin{document}

\title{The Multiplication on $BP$}
\author{Maria Basterra}
\address{Department of Mathematics, University of New Hampshire, Durham, NH}
\email{basterra@math.unh.edu}
\author{Michael A. Mandell}
\address{Department of Mathematics, Indiana University, Bloomington, IN}
\email{mmandell@indiana.edu}
\thanks{The second author was supported in part by NSF grant DMS-0804272}

\date{\draftdate}


\begin{abstract}
$BP$ is an $E_{4}$ ring spectrum.  The $E_{4}$ structure is unique up
to automorphism.
\end{abstract}

\maketitle

\section{Introduction}

For each prime $p$, Brown and Peterson \cite{BPBP} constructed a
spectrum $BP$ to capture the non-Bockstein part of the Steenrod
algebra.  In the decades since, $BP$ has become fundamental
in the study of stable homotopy theory, largely because it is the complex
oriented ring spectrum that has the universal $p$-typical formal group
law.  The complex cobordism spectrum $MU$ has the universal formal
group law, and its localization $MU_{(p)}$ splits as a sum of shifted
copies of $BP$.  Working $p$-locally, most statements involving $MU$
then have equivalent but simpler formulations in terms of $BP$.

One fundamental difference between $BP$ and $MU_{(p)}$ involves the
coherence of the multiplicative structure.  The coherence of the
multiplication on $MU$ is well-understood: As a Thom spectrum, $MU$
has a canonical (even prototypical) structure of an $E_{\infty}$ ring
spectrum \cite{mayeinf}. This structure arises from a geometric model
of $MU$ completely independent of formal group law considerations, and
$BP$ does not admit a corresponding construction.  In fact, recent
work of Johnson and Noel \cite{JohnsonNoel} finds a fundamental
incompatibility between the $E_{\infty}$ ring structure on $MU$ and
$p$-typical orientations.

In this paper, we study the coherence of the multiplication on $BP$.
Previous work by Goerss \cite{GoerssObsMU},
Lazarev \cite{lazarev}, and the authors (unpublished) shows that $BP$
is an  $A_{\infty}$ (or $E_{1}$) ring spectrum.  In this paper, we
show that $BP$ admits an $E_{4}$ ring structure, and that this
structure is unique up to automorphism.

\begin{thm}\label{mainthm}
$BP$ is an $E_{4}$ ring spectrum; the $E_{4}$ ring structure is unique
up to automorphism in the homotopy category of $E_{4}$ ring spectra.
\end{thm}

As part of the argument, we construct an idempotent splitting of
$MU_{(p)}$ as an $E_{4}$ ring spectrum with the $E_{4}$ ring spectrum
$BP$ as the unit summand.  We do not know if this idempotent can be
chosen to be the Quillen idempotent; we hope to return to this
question in a future paper.

One of the main reasons for interest in $E_{\infty}$ structures is
that the category of modules over an $E_{\infty}$ ring spectrum has a
symmetric monoidal smash product.  The main result of the paper
\cite{E1E2E3E4} is that for a symmetric monoidal structure on the
homotopy category of modules, an $E_{4}$ structure suffices.  Thus, applying the
main theorem of \cite{E1E2E3E4} to the previous theorem, we obtain the
following corollary. 

\begin{cor}
The derived category of $BP$-modules has a symmetric monoidal smash
product $\sma_{BP}$. 
\end{cor}

Richter \cite{RichterBP} presents a different approach to coherence of
the multiplication of $BP$ using the $n$-stage hierarchy of Robinson
\cite{RobinsonInvent}, which is a homotopy commutative generalization
of the $A_{n}$ hierarchy of Stasheff \cite{StasheffAn}.  In this
terminology, \cite{RichterBP} shows that $BP$ has a
${(2p^2+2p-2)}$-stage structure (which in particular restricts to an
$A_{2p^2+2p-2}$ but not $A_{\infty}$ structure).  The relationship of
this result to Theorem~\ref{mainthm} and the relationship of the
$n$-stage hierarchy to the $E_{n}$ hierarchy are not understood
at the present time.

\subsection*{Outline}
For the proof of Theorem~\ref{mainthm}, we apply the Postnikov tower
obstruction theory for $E_{4}$ ring spectra pioneered by Kriz
and Basterra~\cite{mbthesis}  in the $E_{\infty}$
case.  We review the details of this obstruction theory in
Sections~\ref{seco1}--\ref{seco2} before applying it to prove the
main theorem in Section~\ref{secmainthm}.  Finally, 
Sections~\ref{secdlss} and~\ref{secsubdiv} prove some of the
properties of the spectral sequences for computing Quillen homology and
cohomology stated in Section~\ref{secmultss}.

\subsection*{Conventions}
Although we state most results in homotopy theoretic terms that are
independent of models, where the proofs require models, we will work
in the following context.  We use the category of EKMM $S$-modules
\cite{ekmm} as the basic category of spectra, and we understand
$S$-algebras and commutative $S$-algebras in this context.  For
$E_{n}$ structures, we use the Boardman-Vogt little $n$-cubes operad
$\LC$.  In this context, for a commutative $S$-algebra $R$, an $E_{n}$
$R$-algebra consists of an $R$-module $A$ and maps of $R$-modules
\[
\LC(m)_{+}\sma_{\Sigma_{m}} \underbrace{(A\sma_{R}\dotsb
\sma_{R}A)}_{m\text{ factors}}\to A
\]
satisfying the usual associativity and unit properties.

\subsection*{Acknowledgments}
The second author would like to thank Andrew Blumberg, Tyler Lawson,
and Jim McClure for helpful comments.

\section{Quillen Cohomology of $E_{n}$ Algebras}\label{seco1}

For the obstruction theory for $E_{n}$ algebra in Section~\ref{seco2},
the obstruction groups take values in certain Quillen cohomology
groups.  In this section, we review the construction of these groups
in the special case in which we apply them.  For this, we fix a
connective commutative $S$-algebra $R$ and a connective commutative
$R$-algebra $H$, and we consider the category of $E_{n}$ $R$-algebras
lying over $H$. In practice (in Section~\ref{seco2}), $H$ will be an
Eilenberg--Mac~Lane spectrum $HZ$ for a commutative $\pi_{0}R$-algebra
$Z$, but we 
do not need to assume that 
here.  As we review in this section, the advantage of this setup is
that Quillen cohomology with coefficients in an $H$-module becomes a
cohomology theory and has a corresponding homology theory, each
satisfying the appropriate analog of the Eilenberg--Steenrod axioms.

\begin{defn}
Let $\ARH$ denote the category of $E_{n}$ $R$-algebras lying over $H$:
An object of $\ARH$ is an $E_{n}$ $R$-algebra $A$ together with a map
of $E_{n}$ $R$-algebras $\epsilon \colon A\to H$; a morphism of $\ARH$
is a map of $E_{n}$ $R$-algebras $A\to A'$ such that the composite map
$A\to A'\to H$ is $\epsilon$.
\end{defn}

For an $H$-module $M$, we make $H\vee M$ into an object of $\ARH$ with
the ``square zero'' multiplication:  The structure maps
\[
\LC(m)_{+}\sma_{\Sigma_{m}}(H\vee M)^{(m)}
\to (H\vee M)^{(m)}/\Sigma_{m} \to H\vee M
\]
are induced by the multiplication on $H$ and the $H$-module structure
on $M$ on the summands with one or fewer factors of $M$ and the
trivial map on the summands with two or more factors of $M$. The
trivial map $M\to *$ induces the map $H\vee M\to H$.

\begin{defn}
For $A$ in $\ARH$, let $\QC^{*}(A;M)=\h\ARH(A,H\vee \Sigma^{*}M)$,
where $\h\ARH$ denotes the homotopy category obtained from $\ARH$ by
formally inverting the weak equivalences.
\end{defn}

We note that $\QC^{*}(A;M)$ is actually a graded abelian group and in
fact a $\pi_{*}H$-module.  The map of $H$-modules $M\vee M\to M$
induces a map in $\ARH$ from $H\vee (M\vee M)$ to $H\vee M$, which
induces the natural addition since $H\vee (M\vee M)$ is isomorphic in
$\h\ARH$ to the product of two copies of $H\vee M$.  Likewise an
element of $\pi_{q}H$ induces a map in the derived category of
$H$-modules $\Sigma^{q} M\to M$, which induces a
map $H\vee \Sigma^{p}M\to H\vee \Sigma^{p-q}M$ in $\h\ARH$, and hence
a natural map $\QC^{p}(A;M)\to \QC^{p-q}(A;M)$.

We can also define relative groups as follows.  For an object $A$ in
$\ARH$, we consider $\ARH\bs A$, the under-category of $A$ in $\ARH$.
We regard $H\vee M$ as an object of $\ARH\bs A$ using the map $A\to
H$.  For $X$ in $\ARH\bs A$, we define the relative Quillen cohomology
group as
\[
\QC^{*}(X,A;M)=\h(\ARH\bs A')(X,H\vee \Sigma^{*}M)
\]
where $A'$ is a cofibrant approximation of $A$.  Using these relative
groups and a connecting morphism $\delta \colon \QC^{*}(A;M)\to
\QC^{*+1}(X,A;M)$ described below, we obtain a cohomology theory.

\begin{thm}
The functors $\QC^{*}(-;M)$ and connecting morphism $\delta$ define a
cohomology theory on the
category $\ARH$ satisfying the following version of the
Eilenberg--Steenrod Axioms:
\begin{enumerate}
\item (Homotopy) If $(X,A)\to (Y,B)$ is a weak equivalence of pairs,
then the induced map $\QC^{*}(Y,B;M)\to \QC^{*}(X,A;M)$ is an isomorphism of
graded abelian groups.
\item (Exactness) For any pair $(X,A)$, the sequence
\[
\cdots \to \QC^{n}(X,A;M)\to \QC^{n}(X;M)\to
\QC^{n}(A;M)\overto{\delta} \QC^{n+1}(X,A;M) 
\to \cdots
\]
is exact.
\item (Excision) If $A$ is cofibrant, $A\to B$ and $A\to X$ are
cofibrations, and $Y$ is the pushout $X\amalg_{A}B$, then the map of
pairs $(X,A)\to (Y,B)$ induces an isomorphism of graded abelian groups
$\QC^{*}(Y,B;M)\to \QC^{*}(X,A;M)$.
\item (Product) If $\{X_{\alpha}\}$ is a set of
cofibrant objects and $X$ is the coproduct, then the natural map
$\QC^{*}(X;M)\to \prod \QC^{*}(X_{\alpha};M)$ is an isomorphism. 
\end{enumerate}
\end{thm}

To construct the connecting homomorphism, prove the previous theorem,
and construct the corresponding homology theory requires the
``based'' version of the construction obtained by working over and
under the same commutative $S$-algebra.  The functor $H\sma_{R}(-)$
takes \EnR-algebras lying over $H$ to \EnH-algebras lying over $H$
and is left adjoint to the forgetful functor.  Writing $H\lsmaR (-)$
for the left derived functor, we get a natural isomorphism
\[
\QC^{*}(A;M)=\h\ARH(A,H\vee \Sigma^{*}M)\iso 
\h\AHH(H\lsmaR A,H\vee \Sigma^{*}M).
\]
In addition, for $f\colon A\to X$, we can then identify
$\QC^{*}(X,A;M)$ as 
\[
\h(\AHH\bs (H\sma_{R}A'))(H\lsmaR X,H\vee \Sigma^{*}M) \iso
\h\AHH(Cf, H\vee \Sigma^{*}M)
\]
where $Cf$ denotes the ``cofiber''
\[
Cf=(H\sma_{R}X')\amalg_{(H\sma_{R}A')}H
\]
formed as the pushout of \EnH-algebras,
where $A'\to X'$ is a cofibration modeling the map $A\to X$.  The
homotopy cofiber of the inclusion of $X'$ in $Cf$ is equivalent to the
suspension of $A'$ in the homotopy category of $\AHH$; this
produces the connecting homomorphism satisfying the Exactness Axiom.
The remaining axioms are clear from the construction.

To construct the corresponding homology theory, we switch from the
category $\AHH$ to the category $\NH$ of ``non-unital'' \EnH-algebras. These
are the algebras of $H$-modules over the operad $\NC$ where
\[
\NC(m)=\begin{cases}
\emptyset&m=0\\\LC(m)&m>0.
\end{cases}
\]
We obtain a functor $K(-)=H\vee(-)$ from $\NH$ to $\AHH$ by attaching
a new unit; this is left adjoint to the functor $I$ from $\AHH$ to
$\NH$ which takes the (point-set) fiber of the augmentation $A\to H$.
Since 
the functor $I$ preserves fibrations and acyclic fibrations, we see
that the adjoint
pair $K,I$ forms a Quillen adjunction. Since we can calculate the
effect on homotopy groups of $K$ on arbitrary non-unital
\EnH-algebras and of $I$ on fibrant objects of $\AHH$, we see
that when $A$ is fibrant, a map of augmented $\LC$-algebras $KN\to A$
is a weak equivalence if and only if the adjoint map $N\to IA$ is a
weak equivalence; in other words, $K,I$ is a Quillen equivalence.

\begin{thm}\label{thmaugqe}
The functors $K$ and $I$ form a Quillen equivalence between the
category of non-unital \EnH-algebras and the category of \EnH-algebras
lying over $H$ (the category of augmented \EnH-algebras).
\end{thm}

The square zero \EnH-algebra $H\vee M$ is $KZM$ for the non-unital
\EnH-algebra $ZM$ given by $M$ with the trivial structure maps.  Thus,
we have the further description of $\QC^{*}(A;M)$ as $\h\NH(\rI
(H\lsmaR A),Z\Sigma^{*}M)$
where $\rI$ denotes the right derived functor of $I$.  The functor $Z$
is a right adjoint, with left adjoint $Q$ the indecomposables functor
defined by the coequalizer
\[
\xymatrix@C-1pc{%
\displaystyle 
\bigvee_{m>0}\LC(m)_{+}\sma_{\Sigma_m}N^{(m)}\mathstrut
\ar[r]<-.5ex>\ar[r]<.5ex>
&N\ar[r]&QN.
}
\]
Here one map is the action map for $N$ and the other map is the
trivial map on the factors for $m>1$ and the map 
\[
\LC(1)_{+}\sma N\to *_{+}\sma N \iso N
\]
on the $m=1$ factor.  Since the functor $Z$ preserves fibrations and
weak equivalences, we see that $Q,Z$ forms a Quillen adjunction.

\begin{thm}\label{thmzq}
The functors $Q$ and $Z$ form a Quillen adjunction between the
category of non-unital \EnH-algebras and the category of $H$-modules.
\end{thm}

Writing $\lQ$ for the left derived functor of $Q$, we get reduced
cohomology and homology theories on non-unital \EnH-algebras as follows.

\begin{defn}\label{defnuqc}
For $N$ a non-unital \EnH-algebra, let 
\begin{align*}
D^{*}(N;M)&=\Ext_{H}^{*}(\lQ N,M)\\
D_{*}(N;M)&=\Tor^{H}_{*}(\lQ N,M).
\end{align*}
\end{defn}

In other words, $D^{*}(N;M)$ and $D_{*}(N;M)$ are the homotopy groups
of the derived function $H$-module and derived smash product as in
\cite[IV\S1]{ekmm}.  We then have as a result that 
\[
\QC^{*}(A;M)\iso D^{*}(\rI(H\lsmaR A);M)
\]
and we make the definition 
\[
\QH_{*}(A;M)=D_{*}(\rI(H\lsmaR A);M)
\]
For $A\to X$, we let 
\[
\QH_{*}(X,A;M)=D_{*}(\rI(Cf);M),
\]
for $Cf$ the cofiber as above.  Using the connecting morphism
$\partial$ arising from 
cofibration sequences in $\NH$, we get a homology theory.

\begin{thm}
The functors $\QH_{*}(-;M)$ and connecting morphism $\partial$ define a
homology theory on the
category $\ARH$ satisfying the following version of the
Eilenberg--Steenrod Axioms:
\begin{enumerate}
\item (Homotopy) If $(X,A)\to (Y,B)$ is a weak equivalence of pairs,
then the induced map $\QH_{*}(X,A;M)\to \QH_{*}(Y,B;M)$ is an isomorphism of
graded abelian groups.
\item (Exactness) For any pair $(X,A)$, the sequence
\[
\cdots \to \QH_{n+1}(X,A;M)\overto{\partial} \QH_{n}(A;M)\to
\QH_{n}(X;M) \to \QH_{n}(X,A;M)
\to \cdots
\]
is exact.
\item (Excision) If $A$ is cofibrant, $A\to B$ and $A\to X$ are
cofibrations, and $Y$ is the pushout $X\amalg_{A}B$, then the map of
pairs $(X,A)\to (Y,B)$ induces an isomorphism of graded abelian groups
$\QH_{*}(X,A;M)\to \QH_{*}(Y,B;M)$.
\item (Sum) If $\{X_{\alpha}\}$ is a set of
cofibrant objects and $X$ is the coproduct, then the natural map
$\bigoplus \QH_{*}(X_{\alpha};M)\to \QH_{*}(X;M)$ is an isomorphism. 
\end{enumerate}
\end{thm}

For convenience in inductive statements, we let 
\begin{align*}
\QC[0]^{*}(A;M)&=\Ext_{R}^{*}(\rI A,M),\\
\QH[0]_{*}(A;M)&=\Tor^{R}_{*}(\rI A,M),
\end{align*}
the $R$-module cohomology and homology with coefficients in $M$ of the
augmentation ideal; this is a cohomology theory and a homology theory,
respectively, on the category of $R$-modules over $H$.  

As we mentioned in the introduction, in our main application, $H$ will
be an Eilenberg--Mac Lane spectrum $HZ$ for a commutative
$\pi_{0}R$-algebra $Z$. 
For any $Z$-module $L$, $HL$ has a unique structure as an
$HZ$-module.  In this context we will abbreviate notation for the
coefficients by writing $L$ in place of $HL$, e.g., 
\begin{align*}
\QC^{*}(A;L)&= \QC^{*}(A;HL),\\
\QH_{*}(A;L)&= \QH_{*}(A;HL).
\end{align*}

\section{Properties of Quillen Homology and Cohomology}
\label{secmultss}

In the previous section, we reviewed the construction of the Quillen
cohomology groups that we use in  the obstruction theory in the next
section.  In this section, we review some of their fundamental
properties that we need to construct the obstruction theory and to
calculate the obstruction groups.

Although our obstruction theory involves the Quillen cohomology
groups, we use the Quillen homology groups to help work with them.
Definition~\ref{defnuqc} provides the precise relationship between
Quillen homology and Quillen cohomology in this context.
Computationally, applying the universal coefficient spectral
sequences of \cite[IV\S4]{ekmm}, we obtain the following universal
coefficient spectral sequences for Quillen homology and cohomology.

\begin{thm}[Universal Coefficient Spectral Sequences]\label{thmuniv}
Let $A$ be an \EnR-algebra lying over $H$ and let $M$ be an
$H$-module.  There are spectral sequences with 
\begin{align*}
E^{2}_{s,t} &= \Tor^{\pi_{*}H}_{s,t}(\QH_{*}(A;H),\pi_{*}M)\\
E_{2}^{s,t} &= \Ext_{\pi_{*}H}^{s,t}(\QH_{*}(A;H),\pi_{*}M)
\end{align*}
converging strongly to $\QH_{*}(A;M)$ and conditionally to
$\QC^{*}(A;M)$, respectively.  For $A\to X$ a map of 
\EnR-algebras lying over $H$, there are spectral sequences with 
\begin{align*}
E_{2}^{s,t} &= \Tor^{\pi_{*}H}_{s,t}(\QH_{*}(X,A;H),\pi_{*}M)\\
E^{2}_{s,t} &= \Ext_{\pi_{*}H}^{s,t}(\QH_{*}(X,A;H),\pi_{*}M)
\end{align*}
converging strongly to $\QH_{*}(X,A;M)$ and conditionally to
$\QC^{*}(X,A;M)$, respectively.
\end{thm}

For computing $\QH_{*}(A;H)$, the main result (Theorem~1.3) of
\cite{bmthh} provides an iterative method.  For $A$ a cofibrant
\EnR-algebra lying over $H$, let $N$ be a cofibrant non-unital
\EnH-algebra with a weak equivalence $KN\to H\sma_{R} A$ of
\EnH-algebras lying over $H$.  By definition then $\QH_{*}(A;H)$ is
$\pi_{*}QN$.  Sections~5 and~6 of \cite{bmthh} study the bar
construction $BKN$ and show that it has the structure of an
\EnH[n-1]-algebra and the reduced construction $\tB N$ (constructed by
$K\tB N\simeq BKN$) has the structure of a non-unital
\EnH[n-1]-algebra.  Iterating the bar construction, we have a weak
equivalence  $\tB^{n}N \simeq \Sigma^{n}QN$. In terms of homotopy
groups, we have
\[
\QH_{*-n}(A;H)=\pi_{*}\tB^{n}N,
\]
or more generally with coefficients,
\begin{align*}
\QH_{*-n}(A;M)&=\Tor^{H}_{*}(\tB^{n}N,M),\\
\QC^{*-n}(A;M)&=\Ext_{H}^{*}(\tB^{n}N,M).
\end{align*}
The filtration on the bar construction then gives us a spectral
sequence for computing the Quillen homology and cohomology from the
homology and cohomology of smash powers of $\tB^{n-1}N$.  We state
this in the following form useful for induction. The extra summand of
$\pi_{*}M$ in the statement derives from the fact that $B\tB^{j}N\simeq
H\vee \tB^{j+1}N$, or equivalently, the convention that $(\tB^{j}
N)^{(0)}=H$ rather than $*$. 

\begin{thm}\label{thmbarssE1}
Let $A$ be an \EnR-algebra lying over $H$ and let $N$ be a cofibrant
non-unital \EnH-algebra with a weak equivalence $KN\to H\sma_{R} A$ of 
\EnH-algebras lying over $H$.  For an
$H$-module $M$ and $j<n$, there is a spectral sequence with 
\[
E^{1}_{s,t} = \Tor^{H}_{t}(\underbrace{\tB^{j}N\sma_{H}\dotsb 
\sma_{H}\tB^{j}N}_{s\text{ factors}},M)
\]
converging strongly to $\pi_{*}M\oplus \QH[j+1]_{*-(j+1)}(A;M)$ and
a spectral sequence with 
\[
E_{1}^{s,t} = \Ext_{H}^{t}(\underbrace{\tB^{j}N\sma_{H}\dotsb 
\sma_{H}\tB^{j}N}_{s\text{ factors}},M)
\]
converging conditionally to $\pi_{*}M\oplus \QC[j+1]^{*-(j+1)}(A;M)$.  
\end{thm}

Under flatness or projectivity hypotheses, we have a good description
of the $E^{2}$ page of the spectral sequence.  

\begin{thm}\label{thmbarssE2}
Let $A$ be an \EnR-algebra lying over $H$ and let $H\to F$ be a map of
commutative $S$-algebras.   For $0\leq j<n$, let $B^{j}_{*}$ be the
associative $\pi_{*}F$-algebra $\pi_{*}F\oplus
\QH[j]_{*-j}(A;F)$.
\begin{enumerate}
\item If $B^{j}_{*}$ is flat over $\pi_{*}F$, then the $E^{2}$ page of
the homological spectral sequence of Theorem~\ref{thmbarssE1} with
coefficients in $F$ is
\[
E^{2}_{s,t}=\Tor_{s,t}^{B^{j}_{*}}(\pi_{*}F,\pi_{*}F).
\]
\item If $B^{j}_{*}$ is projective over
$\pi_{*}F$, then the $E_{2}$ page of
the cohomological spectral sequence of Theorem~\ref{thmbarssE1} is
\[
E_{2}^{s,t}=\Ext^{s,t}_{B^{j}_{*}}(\pi_{*}F,\pi_{*}F).
\]
\end{enumerate}
\end{thm}

These spectral sequences also have multiplicative structures.
For the homological spectral sequence, when $n>1$ and $j<n-1$, the bar
construction $B\subdot(\tB^{j}N)$ is a 
simplicial (partial) non-unital $E_{n-j-1}$-algebra.  The structure map
\[
\LC[n-j-1](2)_{+}\sma (B(\tB^{j}N)\sma B(\tB^{j}N))\to B(\tB^{j}N)
\]
preserves the filtration and so induces an action on the spectral
sequence.  Since $\LC[n-j-1](2)\simeq S^{n-j-2}$, for $j=n-2$, we get
an associative multiplication, and for $j<n-2$ we get an associative
commutative multiplication and a $(n-j-2)$-shifted Lie bracket, which
together give the structure of an $(n-j-2)$-Poisson algebra.  However,
since each $B_{s}(\tB^{j}N)$ is actually a (partial) $E_{n-j}$-algebra, the Lie
bracket from the $E_{n-j-1}$-algebra structure is zero, and so the Lie
bracket in the spectral sequence is zero. The Lie bracket therefore
plays no role in spectral sequence computations here, though we do get
the conclusion that it lowers total filtration in $B^{j+1}_{*}$.  As
we show in Section~\ref{secsubdiv}, the cohomology version of the
spectral sequence always has an algebra structure.

\begin{thm}\label{thmbarssmult}
Let $A$ be an \EnR-algebra lying over $H$ and let $H\to F$ be a map of
commutative $S$-algebras.
\begin{enumerate}
\item 
If $n>1$ and $j<n-1$, the homological spectral sequence of
Theorem~\ref{thmbarssE1} with coefficients in $F$ has a natural multiplication satisfying the
Leibniz rule and converging to the multiplication on the target coming
from the $E_{n-(j+1)}$-algebra structure on the iterated bar
construction.  Under the flatness hypothesis of
Theorem~\ref{thmbarssE2}, the multiplication on the $E^{2}$ page
coincides with the usual multiplication on $\Tor^{B^{j}_{*}}_{*,*}(\pi_{*}F,\pi_{*}F)$
for the commutative algebra $B^{j}_{*}$.
\item 
For all $j<n$, the cohomological spectral sequence of
Theorem~\ref{thmbarssE1} with coefficients in $F$ has a natural multiplication satisfying the
Leibniz rule and converging to the multiplication on the target coming
from the diagonal on the bar construction.  Under the projectivity
hypothesis of Theorem~\ref{thmbarssE2}, the multiplication on the
$E_{2}$ page coincides with the usual (Yoneda) multiplication on
$\Ext^{*,*}_{B^{j}_{*}}(\pi_{*}F,\pi_{*}F)$. 
\end{enumerate}
\end{thm}

We need one more set of results from \cite{bmthh}.  Section~8 
of \cite{bmthh} studies 
the compatibility of power operations on the algebra
and on the bar construction.  We summarize what we need in the
following theorem. 
For this, we write $H_{*}(A;\bF_{p})$ for $\pi_{*}(A\lsmaR
H\bF_{p})=\Tor_{*}^{R}(A,H\bF_{p})$, the ordinary $R$-module homology
of $A$ with coefficients in $\bF_{p}$, and write $\tilde
H_{*}(A;\bF_{p})$ for the kernel of the map $H_{*}(A;\bF_{p})\to
\bF_{p}$ induced by $A\to H\to H\bF_{p}$.

\begin{thm}\label{thmolddl}
Let $F=H\bF_{p}$, and suppose $n\geq
2$ and $1<j<n$.  With notation as in Theorem~\ref{thmbarssE2}, the
map $\tilde H_{*}(A;\bF_{p})\to B^{j}_{*+j}$ preserves the Dyer-Lashof
operations defined for $E_{n-j}$ $H\bF_{p}$-algebras.
\end{thm}

We can generalize the theorem above by studying the interaction of the 
Dyer-Lashof operations with the spectral sequence in
Theorem~\ref{thmbarssE1}.  For $n>2$ and $0\leq j<n-2$, each line in the
bar construction $B(\tB^{j}N)$ is a (partial) $E_{n-j}$-algebra, and so its
$\bF_{p}$ homology admits additive operations $Q^{i}$ and
(for $p>2$) $\beta Q^{i}$ on elements in degree $t$ for $2i<t+n-j-1$
for $p>2$ or $i<t+n-j-1$ for $p=2$.
The operation $\beta^{\epsilon}Q^{i}$ is zero if $2i-\epsilon<t$ for
$p>2$ or $i<t$ for $p=2$ 
\cite[III.3.1]{hinfty}.  These operations therefore act on the
$E^{1}$ page of the homological spectral 
sequence of Theorem~\ref{thmbarssE1} (for $M=H\bF_{p}$); they
commute up to sign with the $d_{1}$ differential since this is an
alternating sum of face 
maps, each of which is a map of (partial) $E_{n-j-1}$-algebras.  (The top
operations commute with the face maps because they are the
restriction operation $\xi_{n-j-1}$ and for $p>2$ the operation
$\zeta_{n-j-1}$ in the $E_{n-j-1}$-algebra
homology operations \cite[III.3.3.(2)]{hinfty}.)
We prove the  following theorem in Section~\ref{secdlss}.  

\begin{thm}\label{thmbardl}
Suppose $n>2$ and $0\leq j<n-2$. For $F=H\bF_{p}$, the spectral sequence
of Theorem~\ref{secmultss} admits operations
\begin{align*}
&\beta^{\epsilon}Q^{i}\colon E^{r}_{s,t}\to E^{r}_{s,t+2i(p-1)-\epsilon}&\epsilon =0,1\qquad&p>2\\
&Q^{i}\colon E^{r}_{s,t}\to E^{r}_{s,t+i}&&p=2
\intertext{
for $r\geq 1$ and 
}
&2i<t+n-j-1 &&p>2\\
&i<t+n-j-1 &&p=2,
\end{align*}
satisfying the following properties:
\begin{enumerate}
\item For $r=1$, $\beta^{\epsilon}Q^{i}$ is the Dyer-Lashof operation
$\beta^{\epsilon}Q^{i}$ 
on $H_{*}(B_{s}(\tB^{j}N);\bF_{p})$.
\item $\beta^{\epsilon} Q^{i}$ (anti)commutes with the differential
$d_{r}\colon E^{r}_{s,t}\to 
E^{r}_{s-r,t+r-1}$ by the formula
$\beta^{\epsilon}Q^{i}d_{r}=(-1)^{\epsilon}d_{r}\beta^{\epsilon}Q^{i}$. 
\item If $x\in E^{1}_{s,t}$ is a permanent cycle and 
$\beta^{\epsilon}Q^{i}$ is defined on $E^{1}_{s,t}$, then 
the permanent cycle $\beta^{\epsilon}Q^{i}x$ represents
$\beta^{\epsilon}Q^{i}x$ in $B^{j+1}_{*}$.  If $s+t>2i-\epsilon$ for $p>2$
or $s+t>i$ for $p=2$, then $\beta^{\epsilon}
Q^{i}x$ is hit by a differential.
\end{enumerate}
\end{thm}

Finally, we have the following Hurewicz Theorem relating Quillen
homology with $R$-module homology.  For an $H$-module $M$, we write
$H_{*}(A;M)$ for $\pi_{*}(A\lsmaR M)=\Tor^{R}_{*}(A,M)$, the
$R$-module homology with coefficients in $M$, and likewise,
$H_{*}(X,A;M)$ for the relative homology. Thinking in terms of
non-unital \EnH-algebras, the unit of the $Q,Z$ adjunction $N\to ZQ N$
induces a natural map from (reduced) homology to Quillen homology.

\begin{thm}[Hurewicz Theorem]\label{thmhurewicz}
Let $M$ be a connective $H$-module and let
$A\to X$ be a map in $\ARH$ that is a $q$-equivalence for $q>0$. Then
the natural map from $H_{*}(X,A;M)$ to 
$\QH_{*}(X,A;M)$ is an isomorphism for $*\leq q+1$, i.e.,
$\QH_{*}(X,A;M)=0$ for $*\leq q$ and $H_{q+1}(X,A;M)\iso \QH_{q+1}(X,A;M)$.
\end{thm}

\begin{proof}
Using the homology version of the Universal Coefficient Theorem, it suffices to
consider the case when $M=H$.  We model the map $H\lsmaR A\to H\lsmaR
X$ as a map of non-unital \EnH-algebras $N\to Y$.  Looking at the
filtration on the bar constructions $\tB N$ and $\tB Y$, we see that
on the $s$-th filtration level, the map $N^{(s)}\to Y^{(s)}$ is
$sq$-connected, and for $s=1$, the cofiber has $(q+1)$-st homotopy
group $H_{q+1}(X,A;H)$. (The zeroth filtration level is trivial.)
Since $q>0$, for $s>1$, $sq\geq q+1$, and so it follows that the map
$\tB N$ to $\tB Y$ is a $(q+1)$-equivalence whose cofiber has
$\pi_{q+2}$ given by $H_{q+1}(X,A;H)$ (via the inclusion of the
cofiber of $\Sigma N\to \Sigma Y$ from filtration level $1$).
Repeating this argument, we get that the map $\tB^{n} N$ to $\tB^{n}
Y$ is a $(q+n)$-equivalence whose cofiber has $\pi_{q+n+1}$ given by
$H_{q+1}(X,A;H)$ (via the inclusion of the cofiber of $\Sigma
\tB^{n-1}N\to \Sigma \tB^{n-1}Y$ from filtration level $1$).
\end{proof}

\section{Obstruction Theory for Connective $E_{n}$ Algebras}\label{seco2}

This section generalizes the work of Kriz and Basterra~\cite{mbthesis}
on Postnikov towers and obstruction theory on $E_{\infty}$ ring
spectra to the context of $E_{n}$ ring spectra.  We continue to work in the
context of $E_{n}$ $R$-algebras for a connective commutative
$S$-algebra $R$.  Now we take $H=HZ$ for $Z$ some commutative
$\pi_{0}R$-algebra, typically $Z=\pi_{0}A$ for the $E_{n}$ $R$-algebra
we are interested in.

We begin with the notion of Postnikov tower.  For a connective
cofibrant $E_{n}$ $R$-algebra $A$ fix $H=H\pi_{0}A$, and note that $A$
has an essentially unique structure of an $E_{n}$ $R$-algebra over
$H$: An easy induction in terms of cells for a cell $E_{n}$
$R$-algebra homotopy equivalent to $A$ shows that $A$ admits a map of
$E_{n}$ $R$-algebras to $H$ which induces the identity map on
$\pi_{0}$, and that this map is unique up to homotopy.  Fixing a
choice of augmentation, we can then consider towers in $\ARH$ under $A$,
\[
A \longrightarrow \cdots \to \ps A{q} \to \ps{A}{q -1} \to \cdots \to \ps{A}{0},
\]
where the map $\ps{A}{0}\to H$ is a weak equivalence,
$\pi_{s}\ps{A}{q}=0$ for $s>q$, and the map $A\to \ps{A}{q}$ induces an
isomorphism on homotopy groups $\pi_{s}$ for $s\leq q$.  We call any
such tower a \term{Postnikov tower for $A$ in $\ARH$},
and we say that two Postnikov towers for $A$ in $\ARH$ are equivalent
if they are weakly 
equivalent as towers in $\ARH\bs A$, i.e., in the category of $E_{n}$
$R$-algebras lying over $H$ and under $A$.
Of course, the underlying tower of $R$-modules of a Postnikov tower in
$\ARH$ is a Postnikov tower of $R$-modules.  We have the following
existence and uniqueness theorem for Postnikov towers in $\ARH$.

\begin{thm}
Let $A$ be a connective cofibrant $E_{n}$ $R$-algebra over
$H=H\pi_{0}A$.  Then $A$ has a Postnikov tower in $\ARH$ and any two
are weakly equivalent.
\end{thm}

\begin{proof}
We can construct $\ps{A}{q}$ by attaching $E_{n}$ $R$-algebra cells to
kill off the higher homotopy groups; this constructs a Postnikov tower
where each structure map $A\to \ps{A}{q}$ is a relative cell complex.
It is easy to see that any Postnikov tower is weakly equivalent to one
where the tower maps $\ps{A}{q}\to \ps{A}{q-1}$ are fibrations, and
then an easy cell by cell argument constructs a weak equivalence of
Postnikov towers
from any one where the structure maps are relative cell complexes to
any one where the tower maps are fibrations.
\end{proof}

The $R$-module Postnikov tower is modeled by a tower
of principal fibrations and this leads to an obstruction theory in the
category of $R$-modules.  Kriz's insight is that the same holds for
$E_{\infty}$ ring spectra, and we argue that the same holds for
$E_{n}$ $R$-algebras.  Since $H$ is the final object in $\ARH$ a
principal fibration in $\ARH$ is one that is the pullback of a
fibration whose source is weakly equivalent to $H$.  To avoid
using cofibrant and fibrant approximation in the statements below, 
for the purposes of this section we will understand 
``homotopy fiber'' in terms of the notion of weak pullback in
$\h\ARH$:  We say 
that $\ps{A}{q}$ (or more accurately, the map $\ps{A}{q}\to
\ps{A}{q-1}$) is the \term{homotopy fiber} of a map
$\ps{A}{q-1}\to H\vee M$ in $\h\ARH$ when the square
\[
\xymatrix@C-1pc{%
\ps{A}{q}\ar[d]\ar[r]&H\ar[d]\\
\ps{A}{q-1}\ar[r]&H\vee M
}
\]
in $\h\ARH$ is a weak pullback square, i.e., given any map from $B$ to
$\ps{A}{q-1}$ in $\h\ARH$ such that the composite to $H\vee M$ factors
through (the unique map to) $H$, there exists a compatible map from
$B$ to $\ps{A}{q}$ in $\h\ARH$.  The map from $B$ to $\ps{A}{q}$ will
generally not be unique: When the homotopy fiber is constructed as the
pullback along a fibration for a choice of point-set model of
the map $\ps{A}{q-1}\to H\vee M$ and a fibration model of the map
$H\to H\vee M$, it is easy to see that the set of choices for
the map $B\to \ps{A}{q}$ (when one exists) has a free transitive
action of $\QC^{-1}(B;M)\iso \QC^{0}(B;\Omega M)$.  The following
theorem in part asserts that a Postnikov tower in $\ARH$ can be
constructed by iterated homotopy fibers of $E_{n}$ 
$R$-algebra $k$-invariants. 

\begin{thm}\label{thmtower}
Let $A$ be a connective cofibrant $E_{n}$ $R$-algebra over
$H=H\pi_{0}A$. Then there exists a Postnikov tower for $A$
in $\ARH$,
\[
A \longrightarrow \cdots \to \ps A{q} \to \ps{A}{q -1} \to \cdots \to \ps{A}{0},
\]
and a sequence of classes $k^{n}_{q} \in
\QC^{q+1}(\ps{A}{q-1};\pi_{q}A)$ such that $\ps{A}{0}\to H$ is
a weak equivalence and 
each $\ps{A}{q}$ is the homotopy fiber of the map $k^{n}_{q}\colon
\ps{A}{q-1}\to H\vee \Sigma^{q+1} H\pi_{q}A$.  
Moreover,
this data has the following 
consistency and uniqueness  properties:
\begin{enumerate}
\item The natural map $\QC^{q+1}(\ps{A}{q-1};\pi_{q}A)\to
H^{q+1}(\ps{A}{q-1};\pi_{q}A)$ takes $k^{n}_{q}$ to the $R$-module
$k$-invariant $k_{q}=k^{0}_{q}$. 
\item If $\{\ps{A}{q},k^{n}_{q}\}$ and $\{\ps{A'}{q},k^{\prime
n}_{q}\}$ are two Postnikov towers for $A$ in $\ARH$ constructed as
iterated homotopy fibers as above, then any 
weak equivalence of Postnikov towers relating them sends $k^{n}_{q}$
to $k^{\prime n}_{q}$.
\end{enumerate}
\end{thm}

As a consequence of the uniqueness properties of the two theorems
above, it makes sense to talk about the Postnikov tower
$\{\ps{A}{q}\}$ and $E_{n}$ $R$-algebra $k$-invariants $k^{n}_{q}$ for
connective $E_{n}$ $R$-algebras $A$ that are not necessarily
cofibrant, constructing them by using a cofibrant approximation.

Before proving the Theorem~\ref{thmtower}, we state the following obstruction
theoretic consequences.  The identification of $\ps{A}{q}$ as a
homotopy fiber gives the following $E_{n}$ $R$-algebra analogue of
the elementary obstruction lemma.

\begin{cor}\label{corobsmapstep}
Let $A$ and $B$ be $E_{n}$ $R$-algebras over $H=H\pi_{0}A$, and let
$f_{q-1}\colon B\to \ps{A}{q-1}$ be a map in $\h\ARH$.  Then $f_{q-1}$
lifts in $\h\ARH$ to a map $f_{q}\colon 
B\to \ps{A}{q}$ if and only if $f_{q-1}^{*}k^{n}_{q}\in
\QC^{q+1}(B;\pi_{q}A)$ is zero; if such a lift exists, then after
choosing one, the set of such lifts is in one-to-one correspondence
with elements of $\QC^{q}(B;\pi_{q}A)$.
\end{cor}

Putting this together for all $q$, we get the following corollary.

\begin{cor}\label{corobsmap}
Let $A$ and $B$ be $E_{n}$ $R$-algebras over $H=H\pi_{0}A$, and let
$f_{q}\colon B\to \ps{A}{q}$ be a map in $\h\ARH$. Then $f_{q}$ lifts
in $\h\ARH$ to a map $f\colon 
B\to A$ when an inductively defined sequence of obstructions $o_{s}\in
\QC^{s+1}(B;\pi_{s}A)$ are zero for all $s\geq q+1$.
\end{cor}

On the other hand, consider the case of an 
$R$-module $A$ with a fixed commutative $\pi_{0}R$-algebra
structure on $\pi_{0}A$ and $\pi_{0}A$-module structure on
$\pi_{*}A$.  Suppose $\ps{A}{q-1}$ has an $E_{n}$ $R$-algebra structure
over $H=H\pi_{0}A$, compatible with the given $\pi_{0}A$-module
structure on $\pi_{*}A$.  The $R$-module $k$-invariant $k^{0}_{q}$
and the augmentation $\ps{A}{q-1}\to H$ induce a map of $R$-modules
\[
\ps{A}{q-1}\to H\vee \Sigma^{q+1}\pi_{q}A
\]
such that the homotopy pullback of the inclusion of $H$ in $H\vee
\Sigma^{q+1}\pi_{q}A$ is weakly equivalent to $\ps{A}{q}$.  This map
is represented by a map in $\h\ARH$ if and only if $k^{0}_{q}$ is in the
image of $\QC^{q+1}(\ps{A}{q-1};\pi_{q}A)$.  Combining this
observation with Theorem~\ref{thmtower} above, we get the
following corollary.  

\begin{cor}\label{corobsstructstep}
Let $A$ be an $R$-module with a fixed commutative $\pi_{0}R$-algebra
structure on $\pi_{0}A$ and $\pi_{0}A$-module structure on $\pi_{*}A$.
An $E_{n}$ $R$-algebra structure on
$\ps{A}{q-1}$ over $H=H\pi_{0}A$ compatible with the $\pi_{0}A$-module
structure on $\pi_{*}A$ extends to a compatible $E_{n}$
$R$-algebra structure on $\ps{A}{q}$ if and only if the $R$-module
$k$-invariant $k^{0}_{q}\in H^{q+1}(\ps{A}{q-1};\pi_{q}A)$ lifts to an
element of $\QC^{q+1}(\ps{A}{q-1};\pi_{q}A)$.
\end{cor}

Since $A$ is the homotopy inverse limit of its Postnikov tower, we get
the following corollary. 

\begin{cor}\label{corobsstruct}
Let $A$ be an $R$-module with a fixed commutative $\pi_{0}R$-algebra
structure on $\pi_{0}A$ and $\pi_{0}A$-module structure on $\pi_{*}A$.
There exists a compatible $E_{n}$ $R$-algebra structure on $A$ if each
$R$-module $k$-invariant $k^{0}_{q}\in
H^{q+1}(\ps{A}{q-1},\pi_{q}A)$ lifts to an element of the inductively
defined group $\QC^{q+1}(\ps{A}{q-1},\pi_{q}A)$ (which is defined
only after lifting $k^{0}_{q-1}$) for all $q\geq 1$.
\end{cor}

We now proceed to prove Theorem~\ref{thmtower}.  It is convenient to
construct the $\ps{A}{q}$ as cofibrant objects and the maps
$\ps{A}{q}\to \ps{A}{q-1}$ as fibrations.  At the bottom level, we
construct $\ps{A}{0}$ by factoring the map $A\to H$ as a cofibration
$A\to \ps{A}{0}$ followed by an acyclic fibration $\ps{A}{0}\to H$.

Assume by induction we have constructed $\ps{A}{q-1}$.  Then the
Hurewicz Theorem for $H$-modules \cite[IV.3.6]{ekmm} and the Hurewicz
Theorem~\ref{thmhurewicz} give us canonical isomorphisms
\begin{align*}
\QH_{q+1}(\ps{A}{q-1},A;H)&\iso H_{q+1}(\ps{A}{q-1},A;H)\\
&\iso \pi_{q+1}(H\sma_{R}\ps{A}{q-1},H\sma_{R}A)
\end{align*}
and tell us that $\QH_{s}(\ps{A}{q-1},A;H)=0$ for $s\leq q$.  The Universal
Coefficient Theorem for $R$-modules \cite[IV.4.5]{ekmm} then gives us
a canonical isomorphism
\[
\pi_{q+1}(H\sma_{R}\ps{A}{q-1},H\sma_{R}A) \iso
\pi_{0}H\otimes_{\pi_{0}R}\pi_{q+1}(\ps{A}{q-1},A) \iso 
\pi_{0}A\otimes_{\pi_{0}R}\pi_{q}A,
\]
and the Universal Coefficient Theorem~\ref{thmuniv} gives us a canonical
isomorphism
\begin{align*}
\QC^{q+1}(\ps{A}{q-1},A;\pi_{q}A)&\iso
\Hom_{\pi_{0}A}(\pi_{q+1}(H\sma_{R}\ps{A}{q-1},H\sma_{R}A),\pi_{q}A)\\
&\iso \Hom_{\pi_{0}A}(\pi_{0}A\otimes_{\pi_{0}R}\pi_{q}A,\pi_{q}A)\\
&\iso \Hom_{\pi_{0}R}(\pi_{q}A,\pi_{q}A).
\end{align*}
We let $\ell^{n}_{q}$ be the element of
$\QC^{q+1}(\ps{A}{q-1},A;\pi_{q}A)$ that corresponds to the identity
element of $\Hom_{\pi_{0}R}(\pi_{q}A,\pi_{q}A)$, and we let
$k^{n}_{q}$ be its image in $\QC^{q+1}(\ps{A}{q-1};\pi_{q}A)$.
Note that by construction, the natural map $\QC^{q+1}(\ps{A}{q-1};\pi_{q}A)\to
H^{q+1}(\ps{A}{q-1};\pi_{q}A)$ takes $k^{n}_{q}$ to
the $R$-module $k$-invariant $k^{0}_{q}$.

Let $M=\Sigma^{q+1}H\pi_{q}A$, and choose a fibrant approximation
$(H\vee M)_{f}$ of $H\vee M$ in $\AHH$.  Since we assumed $A$ and
$\ps{A}{q-1}$ are cofibrant, a representative for
$k^{n}_{q}$ is then a map $\kappa\colon \ps{A}{q-1}\to (H\vee M)_{f}$ in
$\ARH$ and a representative for $\ell^{n}_{q}$ is a homotopy in
$\ARH$ from the composite map $A\to (H\vee M)_{f}$ to the trivial map
$A\to H\to (H\vee M)_{f}$.  Factoring $H\to (H\vee M)_{f}$ as an
acyclic cofibration $H\to H'$ and a fibration $H'\to (H\vee M)_{f}$,
and applying the homotopy lifting property, we get a commutative
diagram
\[
\xymatrix@C-1pc{%
A\ar[r]^-{\lambda}\ar[d]&H'\ar[d]\\
\ps{A}{q-1}\ar[r]_-{\kappa}&(H\vee M)_{f}.
}
\]
We form the homotopy fiber $F$ of $k^{n}_{q}$ as the pullback of
$H'\to (H\vee M)_{f}$ along $\kappa$, and we construct $\ps{A}{q}$ by
factoring the induced map $A\to F$ as a cofibration $A\to \ps{A}{q}$
followed by an acyclic fibration $\ps{A}{q}\to F$.  By construction,
$\pi_{q}\ps{A}{q}$ is canonically isomorphic to $\pi_{q}A$ and the map
$A\to \ps{A}{q}$ induces this canonical isomorphism.  This completes
the inductive step and the 
construction of the Postnikov tower as a tower of iterated homotopy fibers.

The consistency statement~(i) follows from the construction of the
$k^{n}_{q}$ above; thus, it remains to prove the uniqueness
statement~(ii).  A zigzag of weak equivalences of Postnikov towers in
particular restricts to a zigzag of weak equivalences of pairs
\[
(\ps{A}{q-1},A) \simeq (\ps{A'}{q-1},A) 
\]
in $\ARH$, which induces an isomorphism on $\QC^{*}$.  Naturality of
the Hurewicz homomorphism implies that the following diagram commutes
\[
\xymatrix@-1pc{%
\QC^{q+1}(\ps{A'}{q-1},A;\pi_{q}A)\ar[r]^-{\iso}\ar[d]_-{\iso}
&H^{q+1}(\ps{A'}{q-1},A;\pi_{q}A)\ar[d]_-{\iso}\ar[r]^-{\iso}
&\Hom_{\pi_{0}R}(\pi_{q}A,\pi_{q}A)\ar[d]_-{\iso}^-{\id}\\
\QC^{q+1}(\ps{A}{q-1},A;\pi_{q}A)\ar[r]_-{\iso}
&H^{q+1}(\ps{A}{q-1},A;\pi_{q}A)\ar[r]_-{\iso}
&\Hom_{\pi_{0}R}(\pi_{q}A,\pi_{q}A)
}
\]
where the right vertical map is the identity.  The hypothesis that
$\ps{A}{q}$ and $\ps{A'}{q}$ are the homotopy fibers of $k^{n}_{q}$
and $k^{\prime n}_{q}$ imply that the images of $k^{n}_{q}$ and
$k^{\prime n}_{q}$ in $\QC^{q+1}(A;\pi_{q}A)$ are zero, which implies
that $k^{n}_{q}$ is in the image of 
$\QC^{q+1}(\ps{A}{q-1},A;\pi_{q}A)$ and $k^{\prime }$ is in the image
of $\QC^{q+1}(\ps{A'}{q-1},A;\pi_{q}A)$.  Choosing
a lift for each, we can identify the corresponding element of
$\Hom_{\pi_{0}R}(\pi_{q}A,\pi_{q}A)$ as the map on $\pi_{q}$ induced by
the map $A\to \ps{A}{q}$ for $k^{n}_{q}$ and the map $A\to \ps{A'}{q}$
for $k^{\prime n}_{q}$.  Thus, both lifts must correspond to the
identity map.  It follows that the zigzag of weak equivalences sends
$k^{\prime n}_{q}$ to $k^{n}_{q}$.  This completes the proof of
Theorem~\ref{thmtower}.

\section{Proof of the Main Theorem}\label{secmainthm}

We now apply the obstruction theory of the previous section to prove
the main theorem. We fix a prime $p$ and argue by induction up the
Postnikov tower as in Corollary~\ref{corobsstruct}.

As a base case, we have 
\[
\ps{MU_{(p)}}{1}=\ps{MU_{(p)}}{0}\iso H\bZ_{(p)},\qquad \text{and}\qquad 
\ps{BP}{1}=\ps{BP}{0}\iso H\bZ_{(p)}
\]
are $E_{\infty}$ ring spectra and isomorphic.  In particular, we have
maps of $E_{4}$ ring spectra 
\[
MU_{(p)}\to \ps{BP}{1}\qquad \text{and}\qquad
\ps{BP}{1}\to \ps{MU_{(p)}}{1}
\]
over $H\bZ_{(p)}$.
We assume by induction that $\ps{BP}{2n-1}$ is an $E_{4}$ ring
spectrum and that we have constructed maps of $E_{4}$ ring spectra
\[
MU_{(p)}\to \ps{BP}{2n-1}\qquad \text{and}\qquad
\ps{BP}{2n-1}\to \ps{MU_{(p)}}{2n-1}
\]
over $H\bZ_{(p)}$.
We then argue that we can extend the $E_{4}$ ring structure on
$\ps{BP}{2n-1}$ to an $E_{4}$ ring structure on $\ps{BP}{2n+1}$ and
that we can obtain maps of $E_{4}$ ring spectra
\[
MU_{(p)}\to \ps{BP}{2n+1}\qquad \text{and}\qquad
\ps{BP}{2n+1}\to \ps{MU_{(p)}}{2n+1},
\]
over $H\bZ_{(p)}$
which extend the previous pair of maps.  This then will prove that
$BP$ has the structure of an $E_{4}$ ring spectrum.

As part of the inductive argument, we will have to prove the following
lemma about the Quillen homology and cohomology of the Postnikov sections.

\begin{lem}\label{lemmain}
Let $A=\ps{MU_{(p)}}{2n+1}$ or $A=\ps{BP}{2n+1}$.  In degrees $\leq
2n+1$, the $E_{4}$ Quillen homology $\QH[4]_{*}(A;\bZ_{(p)})$ and
$E_{4}$ Quillen cohomology $\QC[4]^{*}(A;\bZ_{(p)})$ are concentrated
in even degrees and torsion free.
\end{lem}

As a consequence of this lemma, we see that the $E_{4}$ Quillen
cohomology of $MU_{(p)}$ and of $BP$ is concentrated in even degrees
for the $E_{4}$ structure on $BP$ constructed above.  If $BP'$ denotes
$BP$ with any $E_{4}$ structure, obstruction theory
(Corollary~\ref{corobsmap}) then says that there exists a map of
$E_{4}$ ring spectra $BP$ to $BP'$ over $H\bZ_{(p)}$.  Since this map
preserves the unit, as a self-map on the underlying spectrum $BP$, it
is an isomorphism on $H^{0}(BP;\bF_{p})$, and so is an isomorphism on
$H^{*}(BP;\bF_{p})$, and so is an isomorphism in the stable category.  The map
$BP\to BP'$ is therefore a weak equivalence of $E_{4}$ ring spectra.
This shows that any two $E_{4}$ ring structures on $BP$ are isomorphic
in the homotopy category of $E_{4}$ ring spectra, finishing the proof
of Theorem~\ref{mainthm}.

Before proceeding with the argument, we note that as a byproduct of
the outline above, we get maps of $E_{4}$ ring spectra $MU_{(p)}\to
BP$ and $BP\to MU_{(p)}$.  (Such maps must also exist by
Lemma~\ref{lemmain} and obstruction theory Corollary~\ref{corobsmap}.)
The composite map $BP\to BP$ is a weak equivalence, and so
pre-composing its inverse with the map $BP\to MU_{(p)}$, the composite
self map of $MU_{(p)}$ is now an idempotent $E_{4}$ ring map,
factoring through $BP$, and splitting off $BP$ as the unit summand.
The construction gives no hint or details about which idempotent map
of $MU_{(p)}$ this is.  We hope to return to this question in a future
paper.

We now begin the proof of the inductive step.  We start with the
inductively constructed map of $E_{4}$ ring spectra $\ps{BP}{2n-1}\to
\ps{MU_{(p)}}{2n-1}$.  Ignoring the $E_{4}$ structure, this map
extends to a map in the stable category from $BP$ to $MU_{(p)}$; since
this extension is an isomorphism on $H^{0}(-;\bF_{p})$, it is split by
a map in the stable category $s\colon MU_{(p)}\to BP$.  Choosing such
a splitting, we have that the spectrum-level $k$-invariant 
\[
k^{0}_{2n}(BP)\in H^{2n+1}(\ps{BP}{2n-1},\pi_{2n}BP)
\]
for $BP$ is the image of the
spectrum-level $k$-invariant 
\[
k^{0}_{2n}(MU_{(p)})\in H^{2n+1}(\ps{MU_{(p)}}{2n-1},\pi_{2n}MU_{(p)})
\]
under the map induced by $\ps{BP}{2n-1}\to \ps{MU_{(p)}}{2n-1}$ and
the map $s_{*}\colon \pi_{*}MU_{(p)}\to \pi_{*}BP$.  Since $MU_{(p)}$
is an $E_{4}$ ring spectrum, its spectrum-level $k$-invariant
$k^{0}_{2n}(MU_{(p)})$ lifts to its $E_{4}$ ring spectrum
$k$-invariant 
\[
k^{4}_{2n}(MU_{(p)})\in \QC[4]^{2n+1}(\ps{MU_{(p)}}{2n-1},\pi_{2n}MU_{(p)}).
\]
The image of this under the $E_{4}$ ring map $\ps{BP}{2n-1}\to
\ps{MU_{(p)}}{2n-1}$ and the map $s_{*}\colon \pi_{*}MU_{(p)}\to
\pi_{*}BP$ provides a lift to
$\QC[4]^{2n+1}(\ps{BP}{2n-1},\pi_{2n}BP)$ of $k^{0}_{2n}(BP)$.  This
constructs $\ps{BP}{2n+1}=\ps{BP}{2n}$ as an $E_{4}$ ring spectrum.

Looking at the construction in the previous paragraph, it does not
follow immediately that the $E_{4}$ ring map $\ps{BP}{2n-1}\to
\ps{MU_{(p)}}{2n-1}$ extends to an $E_{4}$ ring map $\ps{BP}{2n+1}\to
\ps{MU_{(p)}}{2n+1}$.  However, Lemma~\ref{lemmain} stated above
implies that the obstructions in
\[
\QC[4]^{2n+1}(\ps{BP}{2n+1},\pi_{2n}MU_{(p)})\qquad \text{and}\qquad
\QC[4]^{2n+1}(\ps{MU_{(p)}}{2n+1},\pi_{2n}BP)
\]
to extending the $E_{4}$ ring maps
\[
\ps{BP}{2n-1}\to \ps{MU_{(p)}}{2n-1}
\qquad \text{and}\qquad
MU_{(p)}\to \ps{BP}{2n-1}.
\]
to $E_{4}$ ring maps
\[
\ps{BP}{2n+1}\to \ps{MU_{(p)}}{2n+1}
\qquad \text{and}\qquad
MU_{(p)}\to \ps{BP}{2n+1}.
\]
are both zero. Thus, the completion of the inductive step reduces to 
Lemma~\ref{lemmain}. 

The rest of the section is devoted to the proof of
Lemma~\ref{lemmain}.  The role of the inductively hypothesized $E_{4}$
ring map $MU_{(p)}\to \ps{BP}{2n-1}$ is that we need it to prove the
part of the lemma concerned with $BP$.  The part of the lemma
concerned with $MU_{(p)}$ is independent of any facts about $BP$, and
so provides the extension $MU_{(p)}\to \ps{BP}{2n+1}$.  A rigorous
organization would be to prove the result for $MU_{(p)}$ first and
then go back and prove the result for $BP$; however, to avoid needless
repetition, we will do the argument all at once.  In fact, since the
Thom isomorphism gives a weak equivalence of augmented $E_{\infty}$
(and hence $E_{4}$) $H\bZ_{(p)}$-algebras
\[
H\bZ_{(p)}\sma MU \simeq H\bZ_{(p)} \sma BU_{+},
\]
the lemma for $MU$ follows from Singer's computation of the cohomology
of $B^{4}BU$ \cite{SingerBU} (see \cite[4.7]{AHSCube} for an easier
computation of this special case).  Nevertheless, we must go through
the computation for $MU$ to obtain the computation for $BP$.

For the proof of the lemma, we need a fact about the Dyer-Lashof
operations on $MU$.  Let $a_{s}$ denote the polynomial generator in
$H_{2s}(MU;\bF_{p})$, which under the Thom isomorphism corresponds to
the standard generator $b_{s}$ in $H_{2s}(BU;\bF_{p})$.  According to
\cite[IX.7.4.(i)]{lms} and 
\cite[Theorem~6]{KochmanDLBU}, we have
\begin{align*}
Q^{s+1}a_{s} &= a_{s+(s+1)(p-1)} + \text{decomposibles} &&p>2\\
Q^{2s+2}a_{s} &= a_{s+(s+1)} + \text{decomposibles} &&p=2.
\end{align*}
For convenience, abbreviate this operation as $\oQ$: On an element $x$ in
degree $2s$, 
\[
\oQ x= Q^{s+1}x \quad(p>2), \qquad \oQ x=Q^{2s+2}x \quad(p=2);
\]
this is a well-defined operation on the homology of $E_{4}$ ring
spectra.  For $\ps{BP}{2n+1}$, when $n\geq p-1$, we have an element
$\xi_{1}$ (for $p>2$) or $\xi_{1}^{2}$ (for $p=2$) in
$H_{2p-2}(\ps{BP}{2n+1};\bF_{p})$ mapping to the correspondingly named
element in $\aA_{2p-2}=H_{2p-2}(H\bF_{p};\bF_{p})$.  Since the map from
$\ps{BP}{2n+1}$ to $H\bF_{p}$ is an $E_{4}$ ring map, we can compute
the Dyer-Lashof operations on this class from the Dyer-Lashof
operations in $H_{*}(H\bF_{p};\bF_{p})$, computed in
\cite[III.2.2--3]{hinfty}.  In particular, we have that in the range
of dimensions where $H_{*}(\ps{BP}{2n+1};\bF_{p})=H_{*}(BP;\bF_{p})$
(i.e., $*\leq 2n+1$), $H_{*}(\ps{BP}{2n+1};\bF_{p})$ is the polynomial
algebra on generators $\xi_{1}, \oQ \xi_{1}, \oQ^{2}\xi_{1}, \dotsc$
for $p>2$ or $\xi_{1}^{2}, \oQ \xi_{1}^{2}, \oQ^{2}\xi_{1}^{2},
\dotsc$ for $p=2$. 

To take advantage of the Dyer-Lashof operation $\oQ$, we reorganize
the polynomial generators of $H_{*}(MU;\bF_{p})$ and
$H_{*}(\ps{BP}{2n+1};\bF_{p})$ as follows.  For $H_{*}(MU;\bF_{p})$,
let $x^{MU}_{0,0}=a_{p-1}$, and choose elements $x^{MU}_{i,j}$ for
$i\geq 1$, $j\geq 0$ in $H_{2s_{i,j}}(MU;\bF_{p})$ such that
\begin{enumerate}
\item $x^{MU}_{i,j+1}=\oQ x^{MU}_{i,j}$ (so $s_{i,j+1}=p(s_{i,j}+1)-1$), and
\item $H_{*}(MU;\bF_{p})$ is polynomial on $x^{MU}_{i,j}$, $i\geq 0,j\geq 0$.
\end{enumerate}
For example, having chosen $x^{MU}_{0,0},\dotsc,x^{MU}_{k,0}$, then
$x^{MU}_{k+1,0}$ can be chosen as $a_{\ell}$ for the smallest positive
integer $\ell$ not among the numbers $s_{i,j}$ for $0\leq i\leq k,
j\geq 0$.  By slight abuse, we will regard the elements $x^{MU}_{i,j}$ with
$2s_{i,j}\leq 2n+1$ as elements of
$H_{*}(\ps{MU_{(p)}}{2n+1};\bF_{p})$; then we have that in degrees
$\leq 2n+1$, $H_{*}(\ps{MU_{(p)}}{2n+1};\bF_{p})$ is polynomial on the
elements $x^{MU}_{i,j}$ with $i\geq 0,j\geq 0$, and $2s_{i,j}\leq 2n+1$.

For $\ps{BP}{2n+1}$, when $n>p-1$, we take $x^{BP}_{0,0}$ to be $\xi_{1}$
for $p>2$ or $\xi_{1}^{2}$ for $p=2$, and $x^{BP}_{0,j+1}=\oQ
x^{BP}_{0,j}$.  Then in degrees $\leq 2n+1$,
$H_{*}(\ps{BP}{2n+1};\bF_{p})$ is polynomial on the elements
$x^{BP}_{0,j}$ with $j\geq 0$ and $2s_{0,j}\leq 2n+1$ (whether or not
$n>p-1$).  Looking at cohomology, we see that the map $MU_{(p)}\to
H\bF_{p}$ sends $a_{p-1}=x^{MU}_{0,0}$ to $\xi_{1}$ or $\xi_{1}^{2}$,
and so it follows that the given map $MU_{(p)}\to \ps{BP}{2n+1}$ sends
$x^{MU}_{0,0}$ to $x^{BP}_{0,0}$ when $n>p-1$.

We are now ready to start the proof of the lemma.  Let $A$ be one of
$\ps{MU}{2n+1}$ or $\ps{BP}{2n+1}$ and write $x_{i,j}$ for the
generator $x^{MU}_{i,j}$ or $x^{BP}_{i,j}$ above, where in the case of
$BP$, we understand $i=0$.  The lemma concerns Quillen homology and
cohomology with coefficients in $\bZ_{(p)}$ and we will use the Universal
Coefficient Theorem~\ref{thmuniv} to obtain this information from
the Quillen homology with coefficients in $\bF_{p}$.  We compute 
\[
B^{q}_{*}=B^{q}_{*}(A)=\bF_{p}\oplus \QH[q]_{*-q}(A;\bF_{p})
\]
in degrees
$\leq 2n+1$ for $q=1,2,3,4$ inductively as follows.

To compute $B^{1}_{*}$, we apply the bar construction spectral
sequence (Theorem~\ref{thmbarssE1}).  Looking at the $E^{1}$ page and
its differential, we see that the $E^{2}$ page for computing
$B^{1}_{*}$ is the exterior algebra on generators $\sigma x_{i,j}$ (in
bidegree $1,2s_{i,j}$) in internal degrees $t\leq 2n+1$.  Because the
sequence is multiplicative (Theorem~\ref{thmbarssmult}.(i)), there can
be no non-zero differentials starting in these degrees, and therefore also no
non-zero differentials hitting these degrees.  Passing to $E^{\infty}$, we see
that the associated graded of $B^{1}_{*}$ is isomorphic to this
exterior algebra in degrees $\leq 2n+2$.  Because $B^{1}_{*}$ is a
Hopf algebra with its generators (in this degree range) primitive and
in odd degree (by Theorem~\ref{thmbarssmult}.(ii)), it follows that
$B^{1}_{*}$ is this exterior algebra in degrees $\leq 2n+2$.

Applying the bar construction spectral sequence again, we see that
the $E^{2}$ page for computing $B^{2}_{*}$ is the divided
power algebra on generators $\sigma^{2} x_{i,j}$ (in bidegree
$1,2s_{i,j}+1$) for $t\leq 2n+2$.  Equivalently, letting
$\gamma_{p^{k}}(\sigma x_{i,j})$ denote the element of
$E^{2}_{p^{k},2p^{k}(2s_{i,j}+1)}$ represented by
\[
\langle \underbrace{\sigma x_{i,j}\mid\dotsb \mid \sigma
x_{i,j}}_{p^{k}\text{ factors}}\rangle \in B_{p^{k}}(B^{1}_{*})=B^{1}_{*}\otimes \dotsb \otimes B^{1}_{*},
\]
then $\sigma^{2}x_{i,j}=\gamma_{1}(\sigma x_{i,j})$ and $E^{2}_{*,*}$ is
the tensor product of truncated polynomial 
algebras
\[
\bigotimes_{i,j,k} 
\bF_{p}[\gamma_{p^{k}}(\sigma x_{i,j})]/(\gamma_{p^{k}}(\sigma x_{i,j}))^{p}
\]
in internal degrees $t\leq 2n+2$.  Since this is concentrated in even
degrees, nothing in internal degree $t\leq 2n+2$ can be hit by a
non-zero differential, and nothing in total degree $\leq 2n+3$ can have a
non-zero differential.  Thus, in total degrees $\leq 2n+3$,
$E^{2}=E^{\infty}$, and we can identify the associated graded of
$B^{2}_{*}$ in these degrees as the tensor product of truncated
polynomial algebras above.  Applying Theorem~\ref{thmolddl} and
using the fact that $\oQ x_{i,j}=x_{i,j+1}$, we see 
that in $B^{2}_{*}$,  
\[
(\sigma^{2}x_{i,j})^{p}=Q^{s}\sigma^{2}x_{i,j}=\sigma^{2}Q^{s}x_{i,j}=\sigma^{2}x_{i,j+1}
\]
(for $s=s_{i,j}+1$ if $p>2$ or $s=2s_{i,j}+2$ if $p=2$).
In the case when $A=\ps{MU_{(p)}}{2n+1}$, because $A$ is an $E_{\infty}$
ring spectrum, we have all Dyer-Lashof operations, and applying
Theorem~\ref{thmbardl}, we see more generally that  
\[
(\gamma_{p^{k}}(\sigma x_{i,j}))^{p}=\gamma_{p^{k}}(\sigma x_{i,j+1}).
\]
It follows in this case that in degrees $\leq 2n+3$, $B^{2}_{*}$ is
the polynomial algebra on generators $z_{i,k}=\gamma_{p^{k}}(\sigma
x_{i,0})$ for $i\geq 0, k\geq 0$.  In the case when $A=\ps{BP}{2n+1}$,
we use the map $B^{2}_{*}(\ps{MU_{(p)}}{2n+1})\to B^{2}_{*}$.  Since this
map sends $x^{MU}_{0,0}$ to $x_{0,0}$, it sends $x^{MU}_{0,j}$ to
$x_{0,j}$, and it sends $\gamma_{p^{k}}(\sigma x^{MU}_{0,j})$ to
$\gamma_{p^{k}}(\sigma x_{0,j})$.  Thus, it follows again that
\[
(\gamma_{p^{k}}(\sigma x_{0,j}))^{p}=\gamma_{p^{k}}(\sigma x_{0,j+1}).
\]
and that in degrees $\leq 2n+3$, $B^{2}_{*}$ is the polynomial algebra
on generators $z_{0,k}=\gamma_{p^{k}}(\sigma x_{i,0})$ for $k\geq 0$.

We compute $B^{3}_{*}$ just as $B^{1}_{*}$ and see that in degrees
$\leq 2n+4$, it is an exterior algebra on odd degree generators.  The
$E^{2}$ page for computing $B^{4}_{*}$ is concentrated in even degrees
for degrees $\leq 2n+5$, and therefore $B^{4}_{*}$ is as well.  Hence
$\QH[4]_{*}(A;\bF_{p})$ is concentrated in even degrees for
degrees $\leq 2n+1$. It follows that $\QH[4]_{*}(A;\bZ_{(p)})$ is
concentrated in even degrees for degrees $\leq 2n+1$ and is torsion-free
in degrees $<2n+1$ and hence in degrees $\leq 2n+1$.
Finally, we conclude that $\QC[4]^{*}(A;\bZ_{(p)})$ is concentrated in
even degrees and torsion free for degrees $\leq 2n+1$.  (In fact, we
see it is also torsion free in degree $2n+2$.)
This completes the proof of Lemma~\ref{lemmain} and therefore also the
inductive step.

\section{Proof of Theorem~\ref{thmbardl}}\label{secdlss}

This section proves Theorem~\ref{thmbardl}, which explains how the
Dyer-Lashof operations fit into the bar construction spectral
sequence of Theorem~\ref{thmbarssE1}.  The $(j+1)$-times iterated bar
construction on an 
$E_{n}$ algebra is the geometric realization of a simplicial 
(partial) $E_{n-j-1}$ algebra; the spectral sequence with coefficients in
$H\bF_{p}$ arises from the simplicial filtration on the geometric
realization of a simplicial (partial) $E_{n-j-1}$ $H\bF_{p}$-algebra.
Theorem~\ref{thmbardl} is a special case of the
following theorem, which fit the Dyer-Lashof operations into general
spectral sequences of this type.

\begin{thm}\label{thmsimpdl}
Let $A\subdot$ be a proper simplicial (partial) $E_{n}$
$H\bF_{p}$-algebra, and let $A=|A\subdot|$ be its geometric
realization.  The spectral sequence 
\[
E^{1}_{s,t}=\pi_{t}A_{s}
\]
converging strongly to $\pi_{s+t}A$ admits operations
\begin{align*}
&\beta^{\epsilon}Q^{i}\colon E^{r}_{s,t}\to E^{r}_{s,t+2i(p-1)-\epsilon}&\epsilon =0,1\qquad&p>2\\
&Q^{i}\colon E^{r}_{s,t}\to E^{r}_{s,t+i}&&p=2
\intertext{
for $r\geq 1$ and 
}
&2i<t+n-1 &&p>2\\
&i<t+n-1 &&p=2.
\end{align*}
If the $(n-1)$-shifted Lie bracket is zero on each
$\pi_{*}A_{s}$, then there are also ``top'' operations $Q^{(n-1+t)/2}$ and $\beta
Q^{(n-1+t)/2}$ for $p>2$ ($n-1+t$ even) and $Q^{n-1+t}$ for $p=2$ (all
$t$) defined on $E^{r}_{s,t}$.  These operations
satisfy the following properties:
\begin{enumerate}
\item On $E^{1}_{s,*}$, the operation $\beta^{\epsilon}Q^{i}$ is the
Dyer-Lashof operation $\beta^{\epsilon}Q^{i}$  
on $\pi_{*}A_{s}$ or in the case of the top
operations, $\xi_{n-1}$ (for $Q^{(n-1+t)/2}$, $p>2$  or $Q^{n-1+t}$,
$p=2$) or $\zeta_{n-1}$ (for $\beta Q^{(n-1+t)/2}$).
\item $\beta^{\epsilon} Q^{i}$ (anti)commutes with the differential
$d_{r}\colon E^{r}_{s,t}\to 
E^{r}_{s-r,t+r-1}$ by the formula
$\beta^{\epsilon}Q^{i}d_{r}=(-1)^{\epsilon}d_{r}\beta^{\epsilon}Q^{i}$. 
\item If $x\in E^{1}_{s,t}$ is a permanent cycle and 
$\beta^{\epsilon}Q^{i}$ is defined on $E^{1}_{s,t}$, then 
the permanent cycle $\beta^{\epsilon}Q^{i}x$ represents
$\beta^{\epsilon}Q^{i}x$ in $\pi_{*} A$, or when $s=0$ and
$\beta^{\epsilon}Q^{i}$ is a top operation, $\xi_{n-1}$ or
$\zeta_{n-1}$.  If $s+t>2i-\epsilon$ for $p>2$ 
or $s+t>i$ for $p=2$, then $\beta^{\epsilon}
Q^{i}x$ is hit by a differential.
\end{enumerate}
\end{thm}

In the statement ``proper'' means that the union of the degeneracies
is a cofibration of some sort and just ensures that the geometric
realization has the correct homotopy type; if $A\subdot$ is
not proper, the spectral sequence actually converges to the homotopy
groups of the thickened realization.  We can always
arrange for $A\subdot$ to be proper, replacing it with a levelwise
equivalent object if necessary.

For the proof of Theorem~\ref{thmsimpdl}, it is convenient to have
$A\subdot$ be ``Reedy fibrant'', meaning that at each level $s$,
the map from $A_{s}$ to the limit over the face
maps of $A_{i}$, $i<s$ is a fibration (more on this in the proof of
Lemma~\ref{lemreedy} below).  We can 
arrange this without loss of generality by replacing $A\subdot$ by a
weakly equivalent simplicial $E_{n}$ $H\bF_{p}$-algebra if necessary.
We now assume this and begin the proof of Theorem~\ref{thmsimpdl}.

We are taking the sign convention that on $E^{1}_{s,t}$, the $d_{1}$
differential is 
\[
d_{1}=(-1)^{t}\sum_{i=0}^{s}(-1)^{i}\del_{i},
\]
where $\del_{i}$ denotes the $i$-th face map in the simplicial
structure.  Since each face map preserves the $E_{n}$
$H\bF_{p}$-algebra Dyer-Lashof operations and these operations (except 
the top operations) are additive, it follows that the $d_{1}$
differential commutes up to sign with these operations (except the top
operations).  Under the hypothesis that the Lie bracket is zero, the
top operations are also additive and the $d_{1}$ differential commutes
up to sign 
with these as well. It follows that the operations are well-defined
on the $E^{2}$ page.  Since part~(i) of the statement holds by
definition and inductively part~(ii) of the statement for the $d_{r}$
differential shows that the operations are well-defined on the
$E^{r+1}$ page, it suffices to verify parts~(ii) and~(iii).

Each element in $E^{r}_{s,t}$ for $r\geq 2$ is represented by an
element of $E^{1}_{s,t}$ with the property that the face maps 
$\del_{1}$,\dots,$\del_{s}$ are zero -- this subset is the \term{normalized}
$E^{1}$ page.  Now given such an element $x$ that survives to
$E^{r}_{s,t}$, we can represent $x$ as a map of
$H\bF_{p}$-modules
\[
\bar x_{s}\colon S_{c}^{t}\to A_{s},
\]
where $S_{c}^{t}$ denotes a cofibrant $t$-sphere in $H\bF_{p}$-modules,
$S_{c}^{t}\simeq \Sigma^{t}H\bF_{p}$.  Moreover, since we have assumed
that $A\subdot$ is Reedy fibrant, we can choose $\bar x_{s}$ such that for
$1\leq i\leq t$, $\del_{i}\bar x_{s}$ is the (point-set) trivial map
$S_{c}^{t}\to A_{s-1}$; we prove the following lemma at the end of the 
section. 

\begin{lem}\label{lemreedy}
Let $x\in \pi_{t}A_{s}$ such that 
$\del_{i}x=0$ in $\pi_{t}A_{s-1}$ for all $1\leq i\leq s$.  Then $x$ can be represented by a
map $\bar x \colon S_{c}^{t}\to A_{s}$ such that $\del_{i}\circ \bar
x=*$ for all $1\leq i\leq s$.
\end{lem}

We also use a relative version of this result, also proved below.

\begin{lem}\label{lemreedyrel}
Let $\bar x \colon S_{c}^{t}\to A_{s}$ be a map such that
$\del_{i}\circ \bar x=*$ for all $1\leq i\leq s$.  If $\bar x$ is
null-homotopic in $A_{s}$, then there exists a null-homotopy $\bar
y\colon CS_{c}^{t}\to A_{s}$ such that $\del_{i}\circ \bar y=*$ for
all $1\leq i\leq s$. 
\end{lem}

The map $\del_{0}\circ \bar x_{s}\colon S_{c}^{t}\to A_{s-1}$ represents
$\del_{0}x$, which is the $d_{1}$ differential of $x$ up to sign.  Since $x$ is
in $E^{r}_{s,t}$, $r\geq 2$, we must have that $\del_{0}\circ \bar x_{s}$ is
null-homotopic.  Applying the relative version of the lemma, we can
choose a null-homotopy
\[
\bar x_{s-1}\colon CS_{c}^{t}\to A_{s-1}
\]
such that $\del_{i}\circ \bar x_{s-1}=*$ for $1\leq i\leq s-1$.  Since
$\del_{0}\del_{0}=\del_{0}\del_{1}$, the restriction of $\del_{0}\circ
\bar x_{s-1}$ to $S_{c}^{t}\subset CS_{c}^{t}$ is the trivial map, and we can
consider $\del_{0}\circ \bar x_{s-1}$ as a map from $S_{c}^{t+1}$ to
$A_{s-2}$.  This map represents the $d_{2}$ differential of $x$ up to
sign.
If $r>2$, we have that this map is null-homotopic and we can choose
\[
\bar x_{s-2}\colon CS_{c}^{t+1}\to A_{s-2}.
\]
We continue in this way to construct $\bar x_{s}$, \dots, $\bar
x_{m}$, where $m=\max(s-r+1,0)$.  In the case when $r\geq s+1$, $x$ is
a permanent cycle; in the other case, $r\leq s$, we write $\bar y$ for
the map $\del_{0}\circ \bar x_{m}$ viewed as a map from $S_{c}^{t+r-1}$ to
$A_{s-r}$, representing $d_{r}x$ up to sign.

The previous paragraph implicitly constructs a simplicial
$H\bF_{p}$-module $X\subdot$ and a map of simplicial $H\bF_{p}$-modules $X\subdot \to
A\subdot$, where the non-degenerate part of $X_{q}$ is:
\begin{itemize}
\item The trivial $H\bF_{p}$-module $*$ if $q<s-r$ or $q>s$,
\item The $H\bF_{p}$-module $S_{c}^{t+r-1}$ if $q=s-r\geq 0$,
\item The $H\bF_{p}$-module $CS_{c}^{t-q+s-1}$ if $q>s-r$ and $q<s$, or
\item The $H\bF_{p}$-module $S_{c}^{t}$ if $q=s$.
\end{itemize}
On each of these, the maps $\del_{i}$ for $i>0$ are trivial, and the
map $\del_{0}$ is the composite
\[
CS_{c}^{t-q+s-1}\to S_{c}^{t-q+s}\subset  CS_{c}^{t-q+s}=CS_{c}^{t-(q-1)+s-1}
\]
for $s-r+1<q<s$, the inclusion $S_{c}^{t}\to CS_{c}^{t}$ for $q=s$ and
the map $CS_{c}^{t+r-2}\to S_{c}^{t+r-1}$ for $q=s-r+1>0$. 
The maps $\bar x_{q}$ and (when it is defined) $\bar y$ define a map
of simplicial $H\bF_{p}$-modules $\bar x\colon X\subdot \to A\subdot$.

We now split up the case $r\leq s$ (to prove part~(ii)) from the
permanent cycle case $r\geq s+1$ (to prove part~(iii)).
In the case $r\leq s$, the spectral sequence for computing the
homotopy groups of $X\subdot$ becomes zero at $E^{r+1}$, and so the
geometric realization of $X\subdot$ is contractible.  
Since $A\subdot$ is a simplicial $E_{n}$ $H\bF_{p}$-algebra, we get 
a map of simplicial $H\bF_{p}$-modules
\[
\LC(p)_{+} \sma_{\Sigma_{p}}X\subdot^{(p)}\to A\subdot
\]
and we look at the associated spectral sequence.  
Let $a$ denote the fundamental 
class of $S_{c}^{t}$ in $\pi_{t}X_{s}$, and let $b$ denote the
fundamental class of $S_{c}^{t+r-1}$ in $\pi_{t+r-1}X_{s-r}$.  Since
$a$ maps to $x$ in $\pi_{s}A_{t}$ and $b$ maps to $d_{r}x$ up to 
sign, it suffices to show that in the spectral sequence for $\LC(p)_{+}
\sma_{\Sigma_{p}}X\subdot^{(p)}$, the $d_{r}$ differential takes an
operation on $a$ to the corresponding operation on $b$ up to the
appropriate sign as in part~(ii)
of the statement of the theorem.  For this we use the comparison maps
\[
S^{\infty}_{+}\sma_{C_{p}}X\subdot^{(p)}\to 
\LC[\infty](p)_{+}\sma_{C_{p}}X\subdot^{(p)}\to
\LC[\infty](p)_{+}\sma_{\Sigma_{p}}X\subdot^{(p)}
\from \LC(p)_{+} \sma_{\Sigma_{p}}X\subdot^{(p)}
\]
where $C_{p}$ denotes the cyclic group of order $p$.  In the degree range we
are looking at for the operations, the map 
\[
\LC(p)_{+} \sma_{\Sigma_{p}}(S_{c}^{t+r-1})^{(p)}
\to \LC[\infty](p)_{+} \sma_{\Sigma_{p}}(S_{c}^{t+r-1})^{(p)}
\]
is an isomorphism or nearly an isomorphism depending on $n$;
see~\cite[III.5.2--3]{homiter}.  In the degree of the top operations
$\xi_{n-1}$ and $\zeta_{n-1}$, it is an isomorphism, and otherwise, by
working with $n-1$ in place of $n$ if necessary, we can make it an
isomorphism.  This reduces the problem to the study of the spectral
sequence for $\LC[\infty](p)_{+} \sma_{\Sigma_{p}}X\subdot^{(p)}$ and from
there to the study of the spectral sequence for
$S^{\infty}_{+}\sma_{C_{p}}X\subdot^{(p)}$.  The advantage of the
latter spectral sequence is that
$S^{\infty}_{+}\sma_{C_{p}}X\subdot^{(p)}$ is canonically a simplicial
object in the category of CW $H\bF_{p}$-modules, and we can compute the
spectral sequence by looking at the cellular chain complex.  The
remainder of the proof of part~(ii) is an easy argument with the
chain-level operations of \cite{gensteen}, as follows.  

We work with the normalized total complexes of $C_{*}(X\subdot)$
and $C_{*}(S^{\infty}_{+}\sma_{C_{p}}X\subdot^{(p)})$:
An element of $C_{*}(X\subdot)$ is an
element of $C_{t}(X_{s})$ for which the face maps $\del_{i}$ are zero
for $i\geq 1$, with total differential $D=(-1)^{t}\del +d$, where
$\del=\del_{0}$ is the simplicial differential and $d$ is the internal
differential in $C_{*}$.  We then have an isomorphism of spectral
sequences between the spectral sequence of the double complex and
the spectral sequence for computing the homotopy groups of the
simplicial $H\bF_{p}$-modules $X\subdot$ and
$S^{\infty}_{+}\sma_{C_{p}}X\subdot^{(p)}$, consistent with one
commonly used set of sign conventions (compatible with the convention
specified above for $d_{1}$).  Let $a_{s}$ be
the generator of $C_{t}(X_{s})$ corresponding to the cell $S_{c}^{t}$
and $a_{s-q}$ for the generator of $C_{t+q}(X_{s-q})$
corresponding to the cell $CS_{c}^{t+q-1}$.  Under the usual sign
conventions for the cellular chains of
$CS_{c}^{t+q-1}=S_{c}^{t+q-1}\sma I$ (where $I$ has basepoint $1$), we
have that $da_{s-q}$ is $(-1)^{t+q}$ times the class representing the
bottom cell $S_{c}^{t+q-1}$.  This gives us the formula
\[
\del a_{s-q+1}=(-1)^{t+q}da_{s-q}.
\]
Taking
\[
\bar a=a_{s}+a_{s-1}+a_{s-2}+ \dotsb +a_{s-r+1},
\]
we have that
\[
D\bar a=(-1)^{t+r-1}\del a_{s-r+1}
\]
is $(-1)^{t+r-1}$ times the generator of $C_{s-r}(X_{s-r})$
corresponding to the cell $S_{c}^{t+r-1}$.  Since $\bar a_{s}$
represents $a$ and $\del a_{s-r+1}$ represents $b$ in the $E^{1}$ page
of the spectral sequence for $\pi_{*}|X\subdot|$, this tells us that with
these sign conventions, $d_{r}a=(-1)^{t+r-1}b$ in $E^{r}$.
The chain-level operation
$\beta^{\epsilon} Q^{i}$ commutes with $\del$ and
$(-1)^{\epsilon}$-commutes with $d$, adding $\epsilon$ to the parity
of the internal degree; defining
$\beta^{\epsilon}Q^{i}\bar a$ to be the class 
\[
\beta^{\epsilon}Q^{i}a_{s}+\beta^{\epsilon}Q^{i}a_{s-1}+\dotsb
+\beta^{\epsilon}Q^{i}a_{s-r+1},
\]
in $C_{*}(S^{\infty}_{+}\sma_{C_{p}}X\subdot^{(p)})$, 
we get the formula
\[
D\beta^{\epsilon}Q^{i}\bar a
=(-1)^{t+r-1+2i(p-1)-\epsilon}\beta^{\epsilon}Q^{i}\del a_{s-r+1}
=(-1)^{t+r-1+\epsilon}\beta^{\epsilon}Q^{i}\del a_{s-r+1.}
\]
Since $\beta^{\epsilon}Q^{i}a_{s}$ represents $\beta^{\epsilon}Q^{i}a$
and $\beta^{\epsilon}Q^{i}\del a_{s-r+1}$
represents $\beta^{\epsilon}Q^{i}b$  
in the $E^{1}$ page of the spectral sequence for
$\pi_{*}|S^{\infty}_{+}\sma_{C_{p}}X\subdot^{(p)}|$, we see that 
\[
d_{r}(\beta^{\epsilon}Q^{i}a)
= (-1)^{t+r-1+\epsilon}\beta^{\epsilon}Q^{i}b
= (-1)^{\epsilon}\beta^{\epsilon}Q^{i}d_{r}a.
\]
This completes the proof of part~(ii).

In the permanent cycle case $r\geq s+1$, the geometric realization of
$X\subdot$ is homotopy equivalent to an $(s+t)$-sphere
$H\bF_{p}$-module and the cycle 
\[
\bar a=a_{s}+a_{s-1}+a_{s-2}+\dotsb+a_{0}
\]
in the normalized total cellular chain complex (in the notation above)
provides a fundamental class in $H_{s+t}|X\subdot|\iso \pi_{s+t}|X\subdot|$.
We have a map from $X\subdot$ to the constant
simplicial object $S_{c}^{s+t}$ induced by the collapse map
$CS_{c}^{t+s-1}\to S_{c}^{s+t}$ in simplicial degree $0$; looking at normalized total chain
complexes, we see that this map is a weak equivalence, sending $\bar a$
to the fundamental class $c$ of the cell $S_{c}^{s+t}$ in 
$C_{s+t}(S_{c}^{s+t})$. 
Looking at the maps of simplicial $H\bF_{p}$-modules,
\[
\xymatrix@R-1pc{%
\LC(p)_{+} \sma_{\Sigma_{p}}X\subdot^{(p)}\ar[r]\ar[d]_-{\simeq}&A\subdot\\
\LC(p)_{+} \sma_{\Sigma_{p}}(S_{c}^{s+t})^{(p)},
}
\]
to prove part~(iii), it suffices to show that the element in 
\[
\pi_{*}(\LC(p)_{+} \sma_{\Sigma_{p}}|X\subdot|^{(p)}) \iso
\pi_{*}|\LC(p)_{+} \sma_{\Sigma_{p}}X\subdot^{(p)}|
\]
represented by $\beta^{\epsilon}Q^{i}a$ in the $E^{\infty}$ page goes
to the element $\beta^{\epsilon}Q^{i}c$ of 
\[
\pi_{*}(\LC(p)_{+} \sma_{\Sigma_{p}}(S_{c}^{s+t})^{(p)})
\] 
given by $\beta^{\epsilon}Q^{i}$ applied to the
fundamental class of $S_{c}^{s+t}$.  (For the second
assertion of part~(iii), note that when $s+t>2i-\epsilon$ ($p>2$) 
or $s+t>i$ ($p=2$), this element is zero.)
Again we look at chain level operations on the
normalized total cellular chain complex of 
$S^{\infty}_{+}\sma_{C_{p}}X\subdot^{(p)}$.  Now we have that 
\[
\beta^{\epsilon}Q^{i}a_{s}+\beta^{\epsilon}Q^{i}a_{s-1}+\dotsb
+\beta^{\epsilon}Q^{i}a_{0}
\]
is a cycle, representing $\beta^{\epsilon}Q^{i}a$.
Since it maps to $\beta^{\epsilon}Q^{i}c$ in the
cellular chain complex of
$S^{\infty}_{+}\sma_{C_{p}}(S_{c}^{s+t})^{(p)}$, this completes the
proof of part~(iii) and the proof of
Theorem~\ref{thmsimpdl}, assuming Lemmas~\ref{lemreedy}
and~\ref{lemreedyrel}.
 
We now turn to the proof of Lemmas~\ref{lemreedy}
and~\ref{lemreedyrel}.  We begin by reviewing partial matching
objects~\cite[2.3]{DwyerKanStover}, 
which correspond to the limit of the last several faces.  Let
$M^{0}_{0}=M^{1}_{0}=*$; for $s>0$, let $M^{0}_{s}=*$, and for $1\leq
i\leq s+1$, let $M^{i}_{s}$ be the limit of the diagram $\aD^{i}_{s}$
which has objects
\begin{itemize}
\item For each $j$ with $s-i< j\leq s$, a copy of $A_{s-1}$ labeled
$(\del_{j},A_{s-1})$.

\item For each $(j,k)$ with $s-i<j<k\leq s$, a copy of $A_{s-2}$
labeled $(\del_{j}\del_{k}, A_{s-2})$. (We understand $A_{-1}=*$.)
\end{itemize}
and maps
\begin{itemize}
\item For each $(j,k)$ with $s-i<j<k\leq s$, a map $(\del_{k},A_{s-1})\to
(\del_{j}\del_{k},A_{s-2})$ given by the map $\del_{j}\colon A_{s-1}\to A_{s-2}$.

\item For each $(j,k)$ with $s-i<j<k\leq s$, a map $(\del_{j},A_{s-1})\to
(\del_{j}\del_{k},A_{s-2})$ given by the map $\del_{k-1}\colon A_{s-1}\to A_{s-2}$.
\end{itemize}
Inductively, for $s>0$, we can identify $M^{i+1}_{s}$ as the pullback
\[
M^{i+1}_{s}=M^{i}_{s}\times_{M^{i}_{s-1}}A_{s-1}.
\]
We have assumed that $A\subdot$ is Reedy fibrant, and this means that
for every $s$, the map $A_{s}\to M^{s+1}_{s}$ induced by
$\del_{0},\dotsc,\del_{s}$ is a fibration.  Using this, we prove the following proposition.

\begin{prop}
The maps $A_{s}\to M^{i}_{s}$ and $M^{i}_{s}\to M^{i-1}_{s}$ are
fibrations.  For $i<s+1$, they are surjections on homotopy groups.
\end{prop}

\begin{proof}
We work by induction on $s$; the base case $s=0$ is clear.  Assuming
the statements for $s-1$, and using the pullback description of
$M^{i+1}_{s}$, we see that the maps $M^{i}_{s}\to M^{i-1}_{s}$ are
fibrations and for $i<s+1$ surjections on homotopy groups.  Starting
from the Reedy fibrant hypothesis that the map $A_{s}\to M^{s+1}_{s}$
is a fibration, downward induction shows that the maps $A_{s}\to
M^{i}_{s}$ are fibrations.  Using the fact that the maps $M^{i}_{s}\to
M^{i-1}_{s}$ are surjections on homotopy groups for $i<s+1$, we can identify
$\pi_{q}M^{s}_{s}$ as the partial matching object $M^{s}_{s}$ for
$\pi_{q}A\subdot$. The fact that simplicial abelian groups satisfy
the Kan condition implies that the map $\pi_{q}A_{s}\to
\pi_{q}M^{s}_{s}$ is a surjection.  Downward induction now implies
that the maps $\pi_{q}A_{s}\to \pi_{q}M^{i}_{s}$ are surjections for $i<s+1$.
\end{proof}

We can now prove Lemmas~\ref{lemreedy} and~\ref{lemreedyrel}.  For
Lemma~\ref{lemreedy}, we show inductively that we can choose $\bar
x_{i}\colon S_{c}^{t}\to A_{s}$ for $i=0,\dotsc,s$ representing $x$
and having the property that the composite to $M^{i}_{s}$ is the
trivial map; then $\bar x=\bar x_{s}$ satisfies $\del_{i}\bar x=*$ for
$1\leq i\leq s$.  Choose any $\bar x_{0}$ representing $x$; since
$M^{0}_{s}=*$, the composite to $M^{0}_{s}$ is trivial.  Assume by
induction that we have constructed $\bar x_{i-1}$.  Since
$\del_{s-i+1}x=0$, we have that $\del_{s-i+1}\circ \bar x_{i-1}$ is
null-homotopic, and we choose a null-homotopy $a\colon CS_{c}^{t}\to
A_{s-1}$.  Since $\del_{s-k+1}\circ \bar x_{i-1}=*$ for $k<i$, the
composite of $\bar x_{i-1}$ to $M^{i-1}_{s-1}$ is trivial, and so we
can regard the composite of $a$ to $M^{i-1}_{s-1}$ as a map $b$ from
$\Sigma S_{c}^{t}\to M^{i-1}_{s-1}$.  Since the map $A_{s-1}\to
M^{i-1}_{s-1}$ is surjective on homotopy groups, after altering $a$
by adding an appropriate map $\Sigma S_{c}^{t}\to A_{s-1}$ if
necessary, we can arrange for $b$ to be null-homotopic.  Choosing a
null-homotopy $h$ of $b$, and regarding it as a homotopy of $a$ rel
$S_{c}^{t}$, we can lift $h$ to a homotopy of $a$ in $A_{s-1}$
since the map $A_{s-1}\to M^{i-1}_{s-1}$ is a fibration.  Looking at
the other side of the homotopy, we get a map $c\colon
CS_{c}^{t}\to A_{s-1}$ which restricts to the composite of $\bar
x_{i-1}$ on $S_{c}^{t}$ and whose composite map $CS_{c}^{t}\to
M^{i-1}_{s-1}$ is the trivial map.  Using the fact that $M^{i}_{s}\iso
M^{i-1}_{s}\times_{M^{i-1}_{s-1}}A_{s-1}$, the map $c$ together
with the trivial map $CS_{c}^{t}\to M^{i-1}_{s}$ defines a map
$CS_{c}^{t}\to M^{i}_{s}$ which restricts to the composite of $\bar
x_{i-1}$ on $S_{c}^{t}$.  Regarding this as a homotopy from the
composite of $\bar x_{i-1}$ into $M^{i}_{s}$ to the trivial map
$S_{c}^{t}\to M^{i}_{s}$ and using the
fact that the map $A_{s}\to M^{i}_{s}$ is a fibration, we can lift
this to a homotopy $g\colon S_{c}^{t}\sma I_{+}\to A_{s}$ starting at
$\bar x_{i-1}$.  Let $\bar x_{i}\colon S_{c}^{t}\to A_{s}$ be the
other side of this homotopy; by construction, it represents $x$ and
its composite to 
$M^{i}_{s}$ is trivial.  This completes the inductive argument and the
proof of Lemma~\ref{lemreedy}.

The proof of Lemma~\ref{lemreedyrel} is similar: Given $\bar x$, we
argue inductively that we can choose maps $\bar y_{i}\colon
CS_{c}^{t}\to A_{s}$ for $i=0,\dotsc,s$ restricting to $\bar x$ on
$S_{c}^{t}$ and having the properties that the composite to
$M^{i}_{s}$ is the trivial map and that the composite to $M^{s}_{s}$,
viewed as a map from $\Sigma S_{c}^{t}$ to $M^{s}_{s}$ is
null-homotopic.  Choose any $\bar y_{0}$ null-homotopy of $\bar x$;
since $A_{s}\to M^{s}_{s}$ is surjective on $\pi_{t+1}$, after
altering $\bar y_{0}$ by adding a map $\Sigma S_{c}^{t}\to A_{s}$ if
necessary, we can arrange for the composite of $\bar y_{0}$ to
$M^{s}_{s}$ to be null-homotopic as required.  Inductively, having
chosen $\bar y_{i-1}$, we can regard the composite map to $A_{s-1}$ as
a map $z\colon \Sigma S_{c}^{t}\to A_{s-1}$, which is null-homotopic.
As above, we can choose a null-homotopy $a$ whose composite to
$M^{i-1}_{s}$ is the trivial map.  We get a homotopy $\Sigma S_{c}^{t}\sma
I_{+}\to M^{i}_{s}$ starting at the composite map of $\bar y_{i-1}$
and ending at the trivial map; we can choose a lift $g\colon
CS_{c}^{t}\sma I_{+}\to A_{s}$ rel $S_{c}^{t}$ starting at $\bar
y_{i-1}$.  Letting $\bar y_{i}$ be the other side of the homotopy, by
construction it restricts to $\bar x$ on $S_{c}^{t}$, its composite to
$M^{i}_{s}$ is trivial, and its composite to $M^{s}_{s}$ is
null-homotopic.  This completes the induction and the proof of
Lemma~\ref{lemreedyrel}.

\section{Proof of Theorem~\ref{thmbarssmult}.(ii)}\label{secsubdiv}

In this section we prove Theorem~\ref{thmbarssmult}.(ii), verifying
that the cohomological 
spectral sequence of Theorem~\ref{thmbarssE1} is multiplicative.  Following the notation from
Section~\ref{secmultss}, let $N$ be a cofibrant non-unital $E_{n}$
$H$-algebra with a weak equivalence $KN\to H\lsmaR A$, and let
$C=\tB^{n-1}N$, a (partial) non-unital $E_{1}$ $H$-algebra.  The
spectral sequences of Theorem~\ref{thmbarssE1} are induced by the
geometric realization filtration on the bar construction $BC$ and the
multiplication on cohomology is induced by the bar diagonal $BC\to
BC\sma_{H}BC$. We prove the following theorem.

\begin{thm}\label{thmdiagonal}
For a (partial) non-unital $E_{1}$ $H$-algebra $C$, the bar diagonal
$BC\to BC\sma_{H}BC$ preserves filtration, where on the left we use
the geometric realization filtration and on the right we use the
smash product of the geometric realization filtrations.  Moreover, the
induced map on the homotopy groups of the filtration quotients 
\[
\pi_{*}(\Sigma^{n}C^{(n)})\to
\pi_{*}(\bigvee_{j=0}^{n}\Sigma^{n}(C^{(j)}\sma_{H}C^{(n-j)}))
\iso \bigoplus_{j=0}^{n} \pi_{*}(\Sigma^{n}C^{(n)})
\]
is the diagonal map.
\end{thm}

Applying $F_{H}(-,F)$ and using the induced filtration, we get a
multiplicative spectral sequence.  The formula for the map on
filtration quotients above induces the standard multiplication
\[
\pi_{-p-q}F_{H}(C^{(p)},F) \otimes \pi_{-p'-q'}F_{H}(C^{(p')},F)
\to \pi_{-(p+p')-(q+q')}F_{H}(C^{(p+p')},F).
\]
This then proves Theorem~\ref{thmbarssmult}.(ii).

To prove Theorem~\ref{thmdiagonal}, we first recall the construction
of the bar diagonal.  We follow \cite[\S8]{bmthh} and define the bar
diagonal as the composite 
\[
BC \iso |\sd B\subdot C| \to |B\subdot C \sma_{H} B\subdot C|
\iso BC \sma_{H} BC
\]
where $\sd$ denotes the edgewise subdivision construction
of \cite[1.1]{BHM} or \cite[\S4]{GraysonExterior}.
Theorem~\ref{thmdiagonal} follows from the geometric 
properties of the natural isomorphism $|\cdot|\iso |\sd (\cdot)|$, as we
now explain by reviewing the construction in detail.

Recall that for a simplicial object $X$, the edgewise
subdivision $\sd X$ is defined as the simplicial object 
\[
(\sd X)_{n}=X_{2(n+1)-1}
\]
with face map $\del_{i}$ given by $\del_{i}\del_{i+n+1}$ on $X_{2(n+1)-1}$ and
degeneracy map 
$s_{i}$ given by $s_{i}s_{i+n+1}$ on $X_{2(n+1)-1}$.
We note that
$\sd X$ is the diagonal simplicial object of a bisimplicial object
that we will denote by $\Sd X$:  This has $p,q$-simplices given by
\[
(\Sd X)_{p,q}=X_{(p+1)+(q+1)-1}
\]
with face and degeneracy maps $\del_{i}$ and $s_{i}$ given by $\del_{i}$ and
$s_{i}$ in the $p$-direction and by $\del_{i+p+1}$ and $s_{i+p+1}$ in the
$q$-direction. 

These constructions have a formulation in terms of the category of
standard simplices $\DDelta$.  An object of $\DDelta$ is an ordered
set $[n]=\{0,\dotsc,n\}$ and a morphism is a weakly-increasing map.  A
simplicial object is then a contravariant functor out of $\DDelta$ and
a bisimplicial object is a contravariant functor out of $\DDelta \times \DDelta$.
Concatenation (disjoint union) defines a functor $\DDelta \times \DDelta \to \DDelta$,
sending $[p],[q]$ to $[(p+1)+(q+1)-1]$; for a simplicial object $X$,
$\Sd X$ is the bisimplicial object obtain by composing concatenation
with $X$.  The diagonal simplicial object $\sd X$ can also be
described as the composite with $X$ of the doubling functor
$\DDelta\to \DDelta$ (the diagonal followed by concatenation).

The standard $n$-simplex is the simplicial set
$\Delta[n]=\DDelta(-,[n])$, where (in geometric and combinatorial
terms) we view $[n]=\{0,\dotsc,n\}$ as the set of vertices: A
simplicial map from $\Delta[m]$ to $\Delta[n]$ corresponds to the map
$[m]$ to $[n]$ in $\DDelta$ determined by where its vertices go.
We can identify the bisimplicial set $\Sd \Delta[n]$ as the union
\[
(\Delta[0]\times \Delta[n]) \cup (\Delta[1]\times \Delta[n-1]) \cup \dotsb
\cup (\Delta[n]\times \Delta[0]),
\]
glued along boundary pieces
\[
\del_{j} \Delta[j+1]\times \Delta[n-j-1] \sim \Delta[j] \times \del_{0}\Delta[n-j].
\]
We certainly have a map from this union to $\Sd \Delta[n]$ induced by
the order preserving surjections $[j]\amalg [n-j]\to [n]$ (sending the
highest element of $[j]$ and the lowest element of $[n-j]$ to the same
element of $[n]$), and every order
preserving map $[j]\amalg [k]\to [n]$ factors through one of these
(uniquely through the union after making the gluing identifications).
For convenience, we will often write $\Delta[p,q]$ for the
bisimplicial set $\Delta[p]\times \Delta[q]$, the
\term{$(p,q)$-bisimplex}. 

The homeomorphism $|\Delta[n]|\iso ||\Sd \Delta[n]||$ is the
\term{prismatic decomposition} \cite[\S2]{McClureSmith} of
$\Delta[n]$.  The (double) geometric realization of the bisimplex
$||\Delta[j,n-j]||=|\Delta[j]|\times |\Delta[n-j]|$ maps to the
geometric realization of $\Delta[n]$ by
sending the point 
with barycentric coordinates
$(s_{0},\dotsc,s_{j}),(t_{0},\dotsc,t_{n-j})$ to the point with
barycentric coordinates
\[
(s_{0}/2,\dotsc, s_{j-1}/2, s_{j}/2 + t_{0}/2, t_{1}/2,\dotsc, t_{n-j}/2).
\]
It follows from the formula that these maps are compatible with the
gluing and define a continuous map from $||\Sd \Delta[n]||$ to
$|\Delta[n]|$; to see that it is a bijection and therefore a
homeomorphism, we note that the image of $\Delta[j,n-j]$
consists of precisely those points with barycentric coordinates
$(s_{0},\dotsc,s_{n})$ such that
\[
s_{0}+\dotsb +s_{j-1}\leq 1/2\qquad \text{and}\qquad
s_{0}+\dotsb +s_{j}\geq 1/2
\]
(with equality holding for one or the other sum exactly when the point
is on one of the glued boundary pieces).  The inverse homeomorphism 
\[
S_{n}\colon |\Delta[n]|\to ||\Sd \Delta[n]||
\]
is a cellular map for the natural cell structure on the geometric
realization on the left and the double geometric realization on the
right.   This is easy to see 
from a reformulation in terms of \term{summation coordinates}: For
barycentric coordinates $(s_{0},\dotsc,s_{n})$, let 
\[
u_{1}=s_{0}, \quad u_{2}=s_{0}+s_{1},\quad \dotsb, \quad
u_{n}=s_{0}+\dotsb+s_{n-1}
\]
so that $0\leq u_{1}\leq\dotsb \leq u_{n}\leq 1$; the summation
coordinates define a homeomorphism of $|\Delta[n]|$ with the space of
weakly increasing sequences of length $n$ in $[0,1]$.  For
$(s_{0},\dotsc,s_{j}),(t_{0},\dotsc,t_{n-j})$ barycentric coordinates
of a point in $\Delta[j,n-j]=\Delta[j]\times \Delta[n-j]$, we get a weakly
increasing sequence of length $n$ in $[0,1]$ defined by
\begin{equation*}
\begin{split}
&v_{1}=s_{0}/2, \quad v_{2}=(s_{0}+s_{1})/2, \quad\dotsb, 
  \quad v_{j}=(s_{0}+\dotsb +s_{j-1})/2,\\
&v_{j+1}=1/2 + t_{0}/2, 
  \quad \dotsb, \quad v_{n}=1/2 +(t_{0}+\dotsb+t_{n-j-1})/2.
\end{split}
\end{equation*}
Glued points in $|\del_{j} \Delta[j+1]|\times |\Delta[n-j-1]|$ and
$|\Delta[j] \times \del_{0}\Delta[n-j]|$ give the same sequence, and this
formula defines a homeomorphism from $||\Sd \Delta[n]||$ to the space
of weakly increasing sequences of length $n$ in $[0,1]$. For
$\Delta[n]$, the cell structure on the space of sequences corresponds
to sequence entries being zero, one, or equal to the next entry; for
$\Sd \Delta[n]$, the cell structure corresponds to entries being zero,
one, $1/2$, or equal to the next entry.  It follows that the map
$S_{n}$ is cellular.

It is clear from the definition that subdivision takes the 
colimit of simplicial objects to the corresponding colimit of
bisimplicial objects, and it is well known that geometric realization
of simplicial $H$-modules and the double geometric realization of
bisimplicial $H$-modules commutes with colimits.
Since we can write any simplicial $H$-module $X$ as the coequalizer
\[
\xymatrix@C-1pc{%
\relax\displaystyle\bigvee_{f\colon [m]\to [n]} X_{n}\sma \Delta [m]_{+}
\ar[r]<1.75ex>\ar[r]<.75ex>
&\relax\displaystyle\bigvee_{[n]}X_{n}\sma \Delta[n]_{+}\ar[r]<1.25ex>
&X,\vphantom{\displaystyle\bigvee_{[n]}}
}
\]
we can write $|X|$ and $||\Sd X||$ as coequalizers
\begin{gather*}
\xymatrix@C-1pc{%
\relax\displaystyle\bigvee_{f\colon [m]\to [n]} X_{n}\sma |\Delta [m]|_{+}
\ar[r]<1.75ex>\ar[r]<.75ex>
&\relax\displaystyle\bigvee_{[n]}X_{n}\sma |\Delta[n]|_{+}\ar[r]<1.25ex>
&|X|\vphantom{\displaystyle\bigvee_{[n]}}
}\\
\xymatrix@C-1pc{%
\relax\displaystyle\bigvee_{f\colon [m]\to [n]} X_{n}\sma ||\Sd\Delta [m]||_{+}
\ar[r]<1.75ex>\ar[r]<.75ex>
&\relax\displaystyle\bigvee_{[n]}X_{n}\sma ||\Sd\Delta[n]|_{+}\ar[r]<1.25ex>
&||\Sd X||.\vphantom{\displaystyle\bigvee_{[n]}}
}
\end{gather*}
The concrete description of the homeomorphisms $S_{n}\colon |\Delta[n]|\to
||\Sd \Delta[n]||$ above makes it clear that they are
compatible with the maps in the coequalizer and induce an isomorphism
$S\colon |X|\to ||\Sd X||$.  We see from the naturality of the
coequalizer diagrams that $S$ is natural.  

Regarding the coequalizers above as taking place in the category of
filtered $H$-modules (or even doubly filtered $H$-modules for $\Sd
X$), we see that $S$ preserves filtrations since the maps $S_{n}$ do. 
In order to understand the map on filtration quotients, it is useful to
observe how the pieces $X_{n}\sma ||\Sd \Delta[n]||$ fit into the
usual description of the double geometric realization built in terms
of the pieces 
\[
\Sd X_{p,q} \sma ||\Delta[p,q]||_{+}=X_{p+q+1}\sma ||\Delta[p,q]||_{+}.
\]
As indicated above, this $(p,q)$ piece arises from the surjection
\[
[p+q+1] = [p] \amalg [q] \to [p+q],
\]
which in these terms corresponds to the degeneracy $s_{p}$.
Naturality now implies that the $X_{n}\sma ||\Delta[j,n-j||_{+}$ piece
of $X_{n}\sma ||\Sd \Delta[n]||_{+}$ in the coequalizer above maps to
the $X_{n+1}\sma ||\Delta[j,n-j]||_{+}$ piece of the double geometric
realization via the map $s_{j}\colon X_{n}\to X_{n+1}$.
The
induced map on filtration quotients therefore sends the $n$-th
filtration quotient 
\[
X_{n}/(s_{0},\dotsc,s_{n}) \sma |\Delta[n]/\partial|
\]
of $|X|$ to the $n$-th filtration quotient
\begin{multline*}
X_{n+1}/(s_{1},\dotsc,s_{n}) \sma ||\Delta[0,n]/\partial||
\ \vee\ 
X_{n+1}/(s_{0},s_{2},\dotsc,s_{n})\sma ||\Delta[1,n-1]/\partial||
\ \vee\ \dotsb\\ \ \vee\ 
X_{n+1}/(s_{0},s_{1},\dotsc,s_{n-1})\sma ||\Delta[n,0]/\partial||
\end{multline*}
of $||\Sd X||$ 
by the map induced by $S_{n}$ and the degeneracy $s_{j}$ on the
$\Delta[j,n-j]$ summand. 

In the case when $X=BC$, we have 
\begin{gather*}
X_{n}/(s_{0},\dotsc,s_{n})=C^{(n)},\\
X_{n+1}/(s_{0},\dotsc,s_{j-1},s_{j+1},\dotsc,s_{n})
=C^{(j)}\sma_{H} (H\vee C) \sma_{H} C^{(n-j)}
\end{gather*}
and the map is induced by the isomorphism $C^{(n)}\iso
C^{(j)}\sma_{H}H\sma_{H}C^{(n-j)}$.  The bisimplicial map from $\Sd
B\subdot C$ to $B\subdot C\sma_{H} B\subdot C$ in bidegree $p,q$ is
the map 
\[
(H\vee C)^{(p+q+1)}\to (H\vee C)^{(p)}\sma_{H}(H\vee C)^{(q)}
\]
induced by the augmentation $H\vee C\to H$ (using the trivial map
$C\to *$) on the $(p+1)$-st factor.  The composite map $BC\to
BC\sma_{H} BC$ therefore induces on $n$-th filtration quotient the map
\[
C^{(n)}\sma |\Delta[n]/ \partial|\to
C^{(0)}\sma_{H}C^{(n)}\sma ||\Delta[0,n]/\partial||
\vee \dotsb \vee C^{(n)}\sma_{H}C^{(0)}\sma ||\Delta[n,0]/\partial||
\]
induced by $S_{n}$ and the isomorphisms $C^{(n)}\iso
C^{(j)}\sma_{H}C^{(n-j)}$.  On homotopy groups this induces the
diagonal map as in the statement of Theorem~\ref{thmdiagonal}.
This completes the proof of Theorem~\ref{thmdiagonal}.

\begin{rem}
In \cite[\S8]{bmthh}, we used the subdivision isomorphism $|\sd X|\iso
|X|$ defined in
\cite[1.1]{BHM} because of its multiplicative properties (see
\cite[8.1]{bmthh}); a check of formulas shows that the
isomorphism $|X|\iso ||\Sd X||\iso |\sd X|$ constructed above is its
inverse, q.v.~\cite[\S4]{GraysonExterior}.
\end{rem}


\begin{rem}
The work above can also be formulated in terms of 
Drinfeld's description of the geometric realization
\cite{DrinfeldGeoReal}, which writes
\[
|X|=\colim_{F\subset I}X(\pi_{0}(I\setminus F))
\]
as the colimit over the finite subsets $F$ of the unit interval $I=[0,1]$, 
where we regard $\pi_{0}(I\setminus F)$ as a finite ordered set, i.e.,
\[
X(\pi_{0}(I\setminus F))=X_{\#\pi_{0}(I\setminus F)-1}.
\]
The double geometric realization is then
\[
||\Sd X|| = \colim_{F_{1},F_{2}\subset I}
X(\pi_{0}(I\setminus F_{1})\amalg \pi_{0}(I\setminus F_{2})).
\]
The isomorphism $S\colon |X|\to ||\Sd X||$ defined above is induced by
the map of diagrams sending 
$F=\{0\leq u_{1}< u_{2}<\dotsb <u_{n}\leq 1\}$ 
to 
\begin{align*}
F_{1}&=\{0\leq v_{1}<\dotsb <v_{j}<1\}=
  \{0\leq 2u_{1}<\dotsb <2u_{j}< 1\}\\
F_{2}&=\{0\leq w_{1}<\dotsb <w_{n-j}\leq 1\}=
  \{0\leq (2u_{j+1}-1) <\dotsb <(2u_{n}-1)\leq 1 \},
\end{align*}
together with the natural transformation 
\[
X(\pi_{0}(I\setminus F))\to
X(\pi_{0}(I\setminus F_{1}) \amalg \pi_{0}(I\setminus F_{2}))
\]
induced by the map of ordered sets
\begin{multline*}
\pi_{0}([0,1]\setminus \{u_{1},\dotsc,u_{n}\}) \from
\pi_{0}([0,1] \setminus \{u_{1},\dotsc, u_{j}, 1/2, 
  u_{j+1},\dotsc u_{n}\})\\
\iso \pi_{0}( ([0,1]\setminus \{v_{1},\dotsc,v_{j}\})
\amalg ([0,1]\setminus \{w_{1},\dotsc,w_{n-j}\})).
\end{multline*}
After verifying that the induced map is continuous, the remaining
verifications are essentially the same as the ones above written in 
terms of summation coordinates.
\end{rem}



\bibliographystyle{plain}\let\bibcomma\relax

\def\noopsort#1{}\def\MR#1{}
\def\preprint{\ifx\bibcomma\undefined Preprint\else preprint\fi}
\providecommand{\bysame}{\leavevmode\hbox to3em{\hrulefill}\thinspace}
\providecommand{\MR}{\relax\ifhmode\unskip\space\fi MR }
\providecommand{\MRhref}[2]{%
  \href{http://www.ams.org/mathscinet-getitem?mr=#1}{#2}
}
\providecommand{\href}[2]{#2}

\end{document}